\documentclass{amsart}
\usepackage{amsmath,amssymb,latexsym,amsthm,amscd,amsxtra}
\usepackage[all]{xy}
\usepackage{verbatim}
\usepackage{dsfont}
\usepackage{MnSymbol}
\usepackage{tikz}
\usepackage{tikz-cd}
\usepackage{tkz-euclide}
\usepackage{hyperref}
\hypersetup{
    colorlinks=true,
    linkcolor=red,
    filecolor=magenta,      
    urlcolor=cyan,
    citecolor=blue,
}
\usepackage{graphicx}
\usepackage{mathrsfs}
\usepackage{enumitem}
\setlist{nolistsep}
\graphicspath{converted_graphics}
\theoremstyle{plain} 
\newtheorem{theorem}{Theorem}[section]
\newtheorem{lemma}[theorem]{Lemma}
\newtheorem{proposition}[theorem]{Proposition}
\newtheorem{corollary}[theorem]{Corollary}

\newtheorem{conjecture}[theorem]{Conjecture}

\theoremstyle{plain}

\theoremstyle{definition}
\newtheorem{remark}[theorem]{Remark}
\newtheorem{example}[theorem]{Example}
\newtheorem{definition}[theorem]{Definition}
\newcommand{\dsp}{\displaystyle}
\newcommand{\bi}[1]{\textbf{\textit{#1}}}
\DeclareMathOperator{\ssi}{SI}
 
\DeclareMathOperator{\Rep}{Rep(Q,\beta)}
\newcommand{\Hom}{\text{Hom}}

\newcommand{\spl}{\text{SL}}
\newcommand{\gl}{\text{GL}}
\DeclareMathOperator{\dm}{dim}
\newcommand{\C}{\mathbb{C}}
\newcommand{\Z}{\mathbb{Z}}

\newcommand{\F}{\mathscr{F}}
\newcommand{\RR}{\mathbb{R}}
\newcommand{\HH}{\mathbb{H}}

\newcommand{\dt}{\text{det}}
\DeclareMathOperator{\Mat}{Mat}

\newcommand{\unlam}{\underline{\lambda}}
\newcommand{\cone}{C(Q,\beta)}

\DeclareMathOperator{\Id}{Id}

\newcommand{\dmm}{\underline{\textbf{dim}} }

\newcommand{\ie}{\emph{i.e.}~}

\newcommand{\R}{\mathbb{R}}
\newcommand{\CC}{\mathbb{C}}

\newcommand{\xra}{\xrightarrow}

\newcommand{\extp}{\@ifnextchar^\@extp{\@extp^{\,}}}
\def\@extp^#1{\mathop{\bigwedge\nolimits^{\!#1}}}

\DeclareMathOperator{\semi}{SI}

\DeclareMathOperator{\lam}{\lambda}

\newcommand{\si}[1]{\semi(Q,{#1})}

\newcount\cols
{\catcode`,=\active\catcode`|=\active
\gdef\Young(#1){\hbox{$\vcenter
{\mathcode`,="8000\mathcode`|="8000
\def,{\global\advance\cols by 1 &}%
\def|{\cr
      \multispan{\the\cols}\hrulefill\cr
       &\global\cols=2 }%
  \offinterlineskip\everycr{}\tabskip=0pt
  \dimen0=\ht\strutbox \advance\dimen0 by \dp\strutbox
    \halign
    {\vrule height \ht\strutbox depth \dp\strutbox##
      &&\hbox to \dimen0{\hss$##$\hss}\vrule\cr
     \noalign{\hrule}&\global\cols=2 #1\crcr
     \multispan{\the\cols}\hrulefill\cr%
   }
}$}} }

\begin{document}

\title[]{Generalized Littlewood-Richardson coefficients for branching rules of $\gl(n)$ and extremal weight crystals}

\author{Brett Collins}
\thanks{The author was partially supported by the NSA under grant H98230-15-1-0022.}

\begin{abstract}
Following the methods used by Derksen-Weyman in \cite{DW11} and  Chindris in \cite{Chi08}, we use quiver theory to represent the generalized Littlewood-Richardson coefficients for the branching rule for the diagonal embedding of $\gl(n)$ as the dimension of a weight space of semi-invariants. Using this, we prove their saturation and investigate when they are nonzero. We also show that for certain partitions the associated stretched polynomials satisfy the same conjectures as single Littlewood-Richardson coefficients. We then provide a polytopal description of this multiplicity and show that its positivity may be computed in strongly polynomial time. Finally, we remark that similar results hold for certain other generalized Littlewood-Richardson coefficients.
\end{abstract}

\maketitle

\section{Introduction}

\subsection{Context and motivation}

Littlewood-Richardson coefficients appear in many contexts in representation theory, such as the coefficients in the decomposition of a product of symmetric polynomials or the tensor product of irreducible representations of $\gl(n)$. One way to define Littlewood-Richardson coefficients is as follows: let $V$ be a complex vector space of dimension $n$ and $\lambda = (\lambda_1,\ldots, \lambda_n)$ a weakly decreasing sequence of $n$ integers. Denote the irreducible rational representation of $\gl(n)$ with highest weight $\lambda$ by $S^\lambda(V)$. Given three weakly decreasing sequences of $n$ integers $\lambda(1), \lambda(2), \lambda(3)$, the Littlewood-Richardson coefficient $c^{\lambda(2)}_{\lambda(1),\lambda(3)}$ is defined to be the multiplicity of $S^{\lambda(2)}(V)$ in $S^{\lambda(1)}(V)\otimes S^{\lambda(3)}(V)$, that is, 
\[c^{\lambda(2)}_{\lambda(1),\lambda(3)} = \dm_\C \Hom_{\gl(V)}(S^{\lambda(2)}(V), S^{\lambda(1)}(V) \otimes S^{\lambda(3)}(V)).
\]

Similarly, sums of products of these coefficients, which we call generalized Littlewood-Richardson coefficients throughout this paper, appear naturally in the decompositions of various algebraic objects. In particular, generalized Littlewood-Richardson coefficients describe the multiplicities in the branching rules of  restricted representations of $\gl(n)$, as described in \cite{HJ09} and \cite{HTW05}. While there is no known way to describe the multiplicities in each of these branching rules using quiver theory, we show that we can do exactly such for one of them, that is, we describe the coefficients as the dimension of a weight space of semi-invariants for a certain quiver, dimension vector, and weight. More generally, we do this for the generalized Littlewood-Richardson coefficient
\begin{equation}\label{one}
f(\lambda(1),\ldots, \lambda(m)) := 
\sum c^{\lambda(1)}_{\alpha(1),\alpha(2)} c^{\lambda(2)}_{\alpha(2),\alpha(3)}\cdots  c^{\lambda(m-1)}_{\alpha(m-1),\alpha(m)}  c^{\lambda(m)}_{\alpha(m),\alpha(1)}
\end{equation}
for $m \geq 4$ and even, where the summation ranges over all partitions $\alpha(i)$. This multiplicity describes the coefficients arising from the branching rule for the diagonal embedding $\gl(n) \subseteq \gl(n) \times \gl(n)$ in the case $m=6$. We remark in the \hyperlink{section-others}{last section} that similar techniques likewise represent two other generalized Littlewood-Richardson coefficients in this manner.

Recently, Littlewood-Richardson coefficients have been of vital interest in geometric complexity theory which seeks to determine the complexity of computational problems by using tools from algebraic geometry and representation theory to provide lower bounds, and the complexity is quite commonly compared to that of computing certain multiplicities like Littlewood-Richardson coefficients. Understanding the complexity of certain cases of these generalized coefficients or even whether they're nonzero can then be used in comparison to other computational problems. A common technique in combinatorics is to associate a polytope to a multiplicity in such a way that the number of lattice points of the polytope is precisely this number. Because the polytope is defined by a system of linear inequalities, combinatorial optimization may then be used to determine the complexity of the multiplicities as well as the properties of the polytope. 

Knutson and Tao \cite{KT99} provided a polytopal description of Littlewood-Richardson coefficients, allowing them to give a combinatorial proof of the saturation of the coefficients (Theorem \ref{LR-saturation}) and complete the proof of Horn's conjecture (Theorem \ref{Horn}). Derksen and Weyman \cite{DW00a} then reproved the saturation property in the context of quiver representations by using the saturation of weight spaces of semi-invariants. The motivation for this paper may then be summarized as providing an explicit quiver theoretic interpretation of these generalized Littlewood-Richardson coefficients in order to prove their saturation  and use results of quiver theory to study their combinatorial and geometric properties.

\subsection{Main results}
One of the main and most useful results is that of the saturation of this multiplicity.

\begin{theorem}\label{saturation-sun}
Let $\lambda(1),\ldots, \lambda(2k)$ be weakly decreasing sequences of $n$ integers for $k \geq 2$. For every integer $r \geq 1$,
\[
f(r\lambda(1),\ldots, r\lambda(2k)) \neq 0 \Longleftrightarrow
f(\lambda(1),\ldots, \lambda(2k)) \neq 0.
\]
\end{theorem}

We extend the results of \cite{KT99} by using their hive models to provide a polytopal description of the generalized Littlewood-Richardson coefficients in Section \ref{polytope}. By using results in combinatorial optimization theory and the above saturation property, we prove the following theorem.

 \begin{theorem}\label{sun-complexity}
Determining whether the multiplicity \emph{(\ref{one})} is positive or not can be decided in polynomial time. Even more, it can be decided in strongly polynomial time in the sense of \cite{Tar86}.
 \end{theorem}
 
 Horn's conjecture (Theorem \ref{Horn}) relates the set of possible eigenvalues arising from a sum of Hermitian matrices to the nonvanishing of Littlewood-Richardson coefficients. To describe a corresponding statement of the conjecture to multiplicity (\ref{one}), we need to define some notation. For an $m$-tuple  $(I_1,\ldots, I_m)$ of subsets of $\{1,\ldots, n\}$, define the following weakly decreasing sequences of integers (the notation is explained in the \hyperref[notation]{subsection} at the end of this section): 
\[
\unlam(I_i) = \begin{cases}
\lambda'(I_i) & i \text{ even} \\
\lambda'(I_i) -((|I_i| - |I_{i-1}| - |I_{i+1}|)^{n-|I_i|}) & i \text{ odd},
\end{cases}
\]
where we identify $I_0$ and $I_m$. Define the set $K(n,m) \subseteq \R^{mn}$, $m \geq 4$ and even,  to be all $m$-tuples $(\lambda(1),\ldots, \lambda(m))$ of weakly decreasing sequences of $n$ reals that satisfy $\sum_{i \text{ even}} |\lambda(i)| = \sum_{i \text{ odd}} |\lambda(i)|$ and 
		\begin{equation*}
		\sum_{j \in I_i} \sum_{i \text{ even}} \lambda(i)_j \leq \sum_{j \in I_i} \sum_{i \text{ odd}} \lambda(i)_j
		\end{equation*}
		for every tuple $(I_1, \ldots, I_m)$ such that $\unlam(I_i)$, $1 \leq i \leq m$, are partitions and 
		\[
		f(\unlam(I_1), \ldots, \unlam (I_m)) \neq 0,
		\]
This makes $K(n,m)$ a rational convex polyhedral cone in $\R^{mn}$. A corresponding statement of Horn's conjecture for this multiplicity is then as follows, where we describe the generalized eigenvalue problem for $f$ in Section \ref{Hermitian}.

\begin{theorem}
\label{generalization-Horn}
The following statements are true.
	\begin{enumerate}
		\item The cone $K(n,m) \subseteq \RR^{nm}$, $m \geq 4$ and even, consists of all $m$-tuples $(\lambda(1),\ldots, \lambda(m))$ of weakly decreasing sequences of $n$ reals satisfying $\sum_{i \text{ even}} |\lambda(i)|  = \sum_{i \text{ odd}} |\lambda(i)|$ and 
		\[
			\sum_{j \in I_i} \sum_{i \text{ even}} \lambda(i)_j \leq \sum_{j \in I_i} \sum_{i \text{ odd}} \lambda(i)_j
			\]
					for every tuple $(I_1, \ldots, I_m)$ such that the $\unlam(I_i)$, $1 \leq i \leq m$, are partitions and 
		\[
		f(\unlam(I_1), \ldots, \unlam (I_m)) = 1.
		\]
		\item If $(\lambda(1),\ldots, \lambda(m)) \in K(n,m)$, then the tuple satisfies the generalized eigenvalue problem for $f$.
		\item If $\lambda(1),\ldots, \lambda(m)$ are weakly decreasing sequences of $n$-integers, then 
		\[
		(\lambda(1),\ldots, \lambda(m)) \in K(n,m) \Longleftrightarrow f(\lambda(1),\ldots,\lambda(m)) \neq 0.
		\]
		\item $\dm K(n,m) =  mn-1$.
		\end{enumerate}
\end{theorem}

In particular, this provides a recursive procedure for finding all nonzero generalized Littlewood-Richardson coefficients of this type. 
We use this description to describe all facets of the cone of effective weights in the case $n=2,\,m=6$ and find the minimal set of inequalities on the $\lambda(i)$ (see Example \ref{n=2} and the \hyperref[appendix]{Appendix}).

One consequence of this description of the sequences in the cone $K(n,m)$ is the following factorization formula.

\begin{theorem}\label{factorization}
	Let $(\lambda(1), \ldots, \lambda(m)) \in K(n,m) \cap \Z^{mn-1}$. For any  tuple of subsets $I=(I_1, \ldots, I_m)$ of $S=\{1,\ldots, n\}$ satisfying the conditions defining $K(n,m)$, we have the factorization
		\[
			f(\lambda(1),\ldots, \lambda(m)) = f(\lambda(1)^*,\ldots, \lambda(m)^*) \cdot  f(\lambda(1)^\#,\ldots, \lambda(m)^\#),
	\]
	where
	\[
	\lambda(p)^* = (\lambda(p)_{i_{j_1}}, \ldots, \lambda(p)_{i_{j_r}}), \quad I_j = \{i_{j_1}, \ldots, i_{j_r}\}, \qquad \lambda(p)^\# = (\lambda(p)_{\tilde{i}_{j_1}}, \ldots, \lambda(p)_{\tilde{i}_{j_{n-r}}}), \quad S \backslash I_j = (\tilde{i}_{j_1}, \ldots, \tilde{i}_{j_{n-r}}).
	\]
	\end{theorem}

In addition, we investigate the stretched Littlewood-Richardson polynomials $f(N\lambda(1),\ldots, N\lambda(m))$ for certain $\lambda(1),\ldots, \lambda(m)$ in Section \ref{stretched-polynomials}. The tuples we investigate turn out to have the same behavior as the stretched Littlewood-Richardson polynomials for a single coefficient, namely, they satisfy conjectures of King, Tollu, and Toumazet \cite{KTT04} and of Fulton,  providing evidence that the conjectures for the stretched polynomials for a single Littlewood-Richardson coefficient extend to those of generalized Littlewood-Richardson coefficients. Our examples are based on corresponding examples in \cite{Fei15} for the quiver associated to a single Littlewood-Richardson coefficient. As opposed to Fei's examples, ours do not always lie on an extremal ray of the cone of effective weights.

The organization of this paper is as follows. In Section \ref{preliminaries} we provide background on quiver invariant theory and state a certain saturation property for effective weights of quivers proven by Derksen and Weyman \cite{DW00a}. The quiver associated to multiplicity (\ref{one}) is defined in Section \ref{section-saturation} and its saturation property is proven. After recalling more detailed descriptions of the facets of the cone of effective weights for acyclic quivers in Section \ref{section-facets}, we describe the facets of our quiver in Section \ref{section-facets-quivers}, which allows a description of the Horn-type inequalities of the multiplicity in Section \ref{section-Horn}. A moment map description of the  cone associated to our quiver is provided by the generalized eigenvalue problem in Section \ref{Hermitian} while we use the Horn-type inequalities to prove a factorization formula in Section \ref{section-factorization}. In Section \ref{stretched-polynomials}, we explicitly calculate the stretched Littlewood-Richardson polynomials  in certain cases and verify that they share  some of the same properties as the stretched polynomials for single Littlewood-Richardson coefficients. We provide a polytopal description of the multiplicity in Section \ref{polytope} and prove the complexity of computing its positivity. Finally, we discuss in Section \ref{section-others} that our methodology can be used to prove similar results for other generalized Littlewood-Richardson coefficients, in particular, for a multiplicity arising from another branching rule of $\gl(n)$ and a multiplicity related to extremal weight crystals. We state without proof the corresponding main results for these multiplicities.

\subsection{Relation to existing literature}
Horn made his famous \hyperref[Horn]{conjecture} in 1962 \cite{Hor62}, yet the motivation for it goes back to Weyl in 1912 \cite{Wey12}. Weyl was interested in necessary and sufficient inequalities on the eigenvalues of Hermitian matrices such that one matrix was the sum of the other two due to questions in solid mechanics. Many advances were made over the next 50 years (see \cite{Ful97b} for a survey of the history and results pertaining to this problem), and the connection was made between the eigenvalue problem and Schubert calculus, resulting in Horn's conjecture.

The second part of the conjecture provides a recursive process for determining all triples $(I_1,I_2,I_3)$ of subsets of $\{1,\ldots, n\}$ which are necessary to determine if $(\lambda(1),\lambda(2), \lambda(3))$ is such a solution. The first major step in proving the conjecture was made when Klyachko \cite{Kly98} found necessary and sufficient  homogeneous linear inequalities for the eigenvalues. It remained, however, to find a minimal set of inequalities. Klyachko had claimed that these inequalities were independent, but Woodward \cite{AW98} showed that many inequalities were redundant, and later Belkale \cite{Bel01} showed that all the inequalities for which $c^{\lambda(I_2)}_{\lambda(I_1),\lambda(I_3)} > 1$ are redundant, which includes the set found by Woodward. The remaining inequalities would be irredundant by a theorem of Klyachko provided the saturation of Littlewood-Richardson coefficients.

\begin{theorem}[Saturation conjecture]\label{LR-saturation}
For weakly decreasing sequences of $n$ integers $\lambda, \mu, \nu$, $c^{N\nu}_{N \lambda, N\mu} \neq 0$ for some positive $N$ if and only if $c^{\nu}_{\lambda, \mu} \neq 0.$
\end{theorem}

The statement that $c^{\nu}_{\lambda, \mu} \neq 0$ implies $c^{N\nu}_{N\lambda,N\mu} \neq 0$ for all positive integers $N$ follows immediately from the Littlewood-Richardson rule or by these multiplicities forming a semigroup (see \cite{Zel99}).  Knutson, Tao, and Woodward \cite{KTW04}
 used combinatorial gadgets called honeycombs and hive models to prove the saturation conjecture, completing the proof of Horn's conjecture; a similar proof of the conjecture using only hive models is found in \cite{Buc00}. A geometric proof of the saturation conjecture using Schubert calculus was also given by Belkale \cite{Bel06}. Derksen, Schofield, and Weyman \cite{DSW07} showed that the number of subrepresentations of a specific dimension of a given dimension vector for an acylic quiver may be determined using Schubert calculus, which is why Derksen and Weyman defined a specific quiver and dimension vector for which this number was a given Littlewood-Richardson coefficient. In this way, they were able to extend results about Littlewood-Richardson coefficients to the quiver setting and prove the saturation conjecture for Littlewood-Richardson coefficients in \cite{DW00a}  by using results from quiver theory.

The multiplicity (\ref{one}) for the branching rule of the diagonal embedding of $\gl(n)$ was first proven in \cite{Kin71}. The proof may also be found in \cite{HTW05} (see \cite{Koi89} and \cite{HJ09} for further discussion).

\subsection{Notation}\label{notation}
A partition $\lambda$ of length $n$ is a weakly decreasing sequence of $n$ positive integers, denoted $\lambda = (\lambda_1, \ldots, \lambda_n)$. 
We identify two partitions $\lambda$ and $\mu$ if one can be written as the other by adjoining (finitely) many zeros at the end of the sequence. As such, we say $\lambda$ is a partition of at most $n$ (nonzero) parts if $\lambda = (\lambda_1, \ldots, \lambda_n) \in \Z^n$ with $\lambda_1 \geq \ldots \geq \lambda_n \geq 0$. If $\lambda = (\lambda_1, \ldots, \lambda_n)$ and $\mu$ are weakly decreasing sequences (not necessarily of integers), we define $r \lambda = (r\lambda_1, \ldots, r\lambda_n)$ for $r \in \RR^+$ and $\lambda + \mu$ is defined by adding componentwise after extending the sequences by zeros as necessary. Every partition $\lambda$ may be identified with a Young diagram, and we denote the conjugate partition as $\lambda'$, which is the partition associated to the reflection of the Young diagram of $\lambda$ across the main diagonal, i.e., switching the rows and columns. If $I=\{z_1 < \ldots < z_r\}$ is an $r$-tuple of integers, $\lambda(I)$ is defined by $\lambda(I) = (z_r-r, \ldots, z_1-1)$. For an integer $c$ and a non-negative integer $r$, we denote the $r$-tuple $(c,\ldots, c)$ more simply as $(c^r)$. For a sequence of real numbers $\lambda = (\lambda_1, \ldots, \lambda_n)$, we define $|\lambda| = \sum_{i=1}^n \lambda_i$. 

For a complex vector space $V$ of dimension $n$ and a weakly decreasing sequence of $n$ integers $\lambda = (\lambda_1, \ldots, \lambda_n)$, we denote the irreducible rational representation of $\gl(V)$ with highest weight $\lambda$ as $S^\lambda(V)$. Given any three weakly decreasing sequences of $n$ integers $\lambda(1), \lambda(2), \lambda(3)$, the Littlewood-Richardson coefficient $c_{\lambda(1), \, \lambda(3)}^{\lambda(2)}$ is defined to be 
$$c_{\lambda(1), \, \lambda(3)}^{\lambda(2)} = \dm_\C \Hom_{\gl(V)}(S^{\lambda(2)}(V), S^{\lambda(1)}(V) \otimes S^{\lambda(3)}(V)),$$
that is, the multiplicity of $S^{\lambda(2)}(V)$ in $S^{\lambda(1)}(V) \otimes S^{\lambda(3)}(V)$.

\vspace{.2in}

\section{Preliminaries}\label{preliminaries}
\subsection{Preliminaries}
A quiver $Q = (Q_0,Q_1,t,h)$ consists of a finite set of vertices $Q_0$, a finite set of arrows $Q_1$, and functions $t,h : Q_1 \to Q_0$ that assign the tail $ta$ and head $ha$ of each arrow $a$, commonly denoted $ta \xra{a} ha$. Note that we allow multiple arrows between two vertices and loops in the directed graph $Q$. 

Throughout this paper we always work over the complex numbers $\C$. A representation $V$ of $Q$ is a family of finite-dimensional vector spaces (over $\C$) $\{V(x) \mid x \in Q_0\}$ together with a family of linear transformations $\{V(a):V(ta) \to V(ha) \mid a \in Q_1\}$. For a representation $V$, its dimension vector $\dmm \, V$ is defined by $\dmm \, V(x) = \dm_\C V(x)$ for all $x \in Q_0$. 
The dimension vectors of representations of $Q$ then lie in $\Gamma = \Z^{Q_0}$, the set of integer-valued functions on $Q_0.$
For each vertex $x \in Q_0$, there is a simple representation $S_x$ defined by the dimension vector $e_x(y) = \delta_{x,y}$ for all $y \in Q_0$, where $\delta_{x,y}$ is the Kronecker delta.

Given two representations $V$ and $W$ of $Q$, define a morphism $\phi : V \to W$ of representations to be a collection of linear maps $\{\phi(x): V(x) \to W(x) \mid x \in Q_0\}$ such that for every arrow $a \in Q_1$ we have $\phi(ha)V(a) = W(a) \phi(ta)$, meaning the diagram
\[
\xymatrix{ 
V(ta) \ar[r]^{\phi(ta)} \ar[d]_{V(a)} & W(ta) \ar[d]^{W(a)} \\
V(ha) \ar[r]^{\phi(ha)} & W(ha)
}
\]
commutes. Define $\Hom_Q(V,W)$, or simply $\Hom(V,W)$, to be the $\C$-vector space of all morphisms from $V$ to $W$. We thus obtain the abelian category $\text{Rep}(Q)$ of all quiver representations of $Q$. We call $V'$ a subrepresentation of $V$ if $V'(x)$ is a subspace of $V(x)$ for all vertices $x \in Q_0$ and $V'(a) = V(a)|_{V'(ta)}$ for all arrows $a \in Q_1$.

For any $\alpha, \beta \in \Gamma$, define the Euler form by 
\[
\langle \alpha, \beta \rangle = \dsp \sum_{x \in Q_0} \alpha(x) \beta(x) - \sum_{a \in Q_1} \alpha(ta) \beta(ha).
\]

\subsection{Semi-invariants for quivers}
For a dimension vector $\beta$ of a quiver $Q$, the representation space of $\beta$-dimensional representations of $Q$ is defined as 
\[
\Rep = \bigoplus_{a \in Q_1} \Hom \left( \C^{\beta(ta)}, \C^{\beta(ha)} \right).
\]
Note that this is simply an affine space. If $\gl(\beta) = \prod_{x \in Q_0} \gl(\beta(x))$, then there is a natural action of $\gl(\beta)$ on $\Rep$ given by simultaneous conjugation: for $g = (g(x))_{x \in Q_0} \in \gl(\beta)$ and $V = \{V(a)\}_{a \in Q_1} \in \Rep$, $g \cdot V$ is defined by 
\[
(g \cdot V)(a) = g(ha)V(a)g(ta)^{-1} \quad \forall a \in Q_1.
\]
Hence, $\Rep$ is a rational representation of the linearly reductive group $\gl(\beta)$ and the $\gl(\beta)$-orbits parameterize the isomorphism classes of $\beta$-dimension representations of $Q$ since the action is simply base change (with respect to a specified basis). If $Q$ is without oriented cycles, there is only one closed $\gl(\beta)$-orbit in $\Rep$ (specifically, the orbit of the unique $\beta$-dimensional semisimple representation $\bigoplus_{x \in Q_0} S_x^{\beta(x)})$, so the invariant ring $\C[\Rep]^{\gl(\beta)}$ is simply $\C$. However, while there are only constant $\gl(\beta)$-invariant polynomial functions on $\Rep$, the action descends to that of the subgroup $\spl(\beta)$, and the invariant ring under the action of this group is highly nontrivial.

Let $\si{\beta} = \C[\Rep]^{\spl(\beta)}$ be the ring of semi-invariants. Since $\gl(\beta)$ is linearly reductive and $\spl(\beta)$ is the commutator subgroup of $\gl(\beta)$, we have the weight space  decomposition
\[
\si{\beta} = \bigoplus_{\sigma \in X^*(\gl(\beta))} \si{\beta}_\sigma,
\]
where $X^*(\gl(\beta))$ is the group of rational characters of $\gl(\beta)$ and 
\[
\si{\beta}_\sigma = \{f \in \C[\Rep] \mid g \cdot f = \sigma(g)f \; \; \forall g \in \gl(\beta)\}
\]
is the space of semi-invariants of weight $\sigma$. A character (or weight) of $\gl(\beta)$ is of the form
\[
\{g(x) \mid x \in Q_0\} \in \gl(\beta) \mapsto \prod_{x \in Q_0} (\det g(x))^{\sigma(x)}
\]
for $\sigma(x) \in \Z$ for all $x \in Q_0$, so we may identify $X^*(\gl(\beta))$ with $\Z^{Q_0}$. For an integer-valued function $\alpha$ on $Q_0$, define $\sigma = \langle \alpha, \cdot \rangle$ by 
\[
\sigma(x) = \langle \alpha, e_x \rangle = \alpha(x) - \sum_{y \to x} \alpha(y), \quad \forall x \in Q_0.
\]
One can similarly define $\sigma = \langle \cdot, \alpha \rangle.$

Given a quiver $Q$ and dimension vector $\beta$, define the set $\Sigma(Q,\beta)$ to be the set of (integral) effective weights:
\[
\Sigma(Q,\beta) = \{\sigma \in \Z^{Q_0} \mid \si{\beta}_\sigma \neq 0\}.
\]
Schofield  \cite{Sch91} constructed distinguished semi-invariants for quivers that proved to be quite useful in studying the ring of semi-invariants. Derksen and Weyman \cite{DW00a} (see also \cite{SB01}) showed that these semi-invariants in fact span all spaces of semi-invariants. An important consequence of this result is the following saturation property.

\begin{theorem}[\cite{DW00a}, Theorem 3]
\label{saturation}
If $Q$ is a quiver without oriented cycles and $\beta$ is a dimension vector, then the set 
\[
\Sigma(Q,\beta) = \{\sigma \in \Z^{Q_0} \mid \si{\beta}_\sigma \neq 0\}
\]
is saturated, that is, if $\sigma$ is a weight and $r \geq 1$ an integer,
\[
\si{\beta}_\sigma \neq 0 \Longleftrightarrow \si{\beta}_{r \sigma} \neq 0.
\]
\end{theorem}

We will later use this theorem to prove the saturation of the multiplicity (\ref{one}) and give an explicit description of the nonzero generalized Littlewood-Richardson coefficients of this form.

\vspace{.2in}

\section{Representing multiplicity (\ref{one}) as the dimension of a weight space of semi-invariants}\label{section-saturation}
 
\subsection{Saturation theorem}
In this section we will show that the multiplicity (\ref{one})  described in the  branching rule for the diagonal embedding for $\gl(n)$ stated in the introduction arises as the dimension of the weight space of semi-invariants for a certain quiver and dimension vector which we construct. A proof of the saturation of the multiplicity will then follow from Theorem \ref{saturation}.

\subsection{Sun quiver}

Construct a quiver $Q$ in the following way: for $k \geq 2$, start with a regular $2k$-gon with the vertices labeled $(n,i)$, $1 \leq i \leq 2k$, which we call the central vertices, and an arrow connecting $(n,i)$ with $(n,i+1)$ (we will always consider $(n,2k+1) = (n,1)$), where the arrows alternate in direction. At each central vertex $(n,i)$ attach an equioriented $A_n$ quiver, called a flag and denoted $\F(i)$, where each $A_n$ is directed the same way as the arrows at the central vertex $(n,i)$. We will later associate each flag with a weakly decreasing sequence with at most $n$ parts and each central arrow with some other partition. For instance, for $k=3$ the quiver looks like
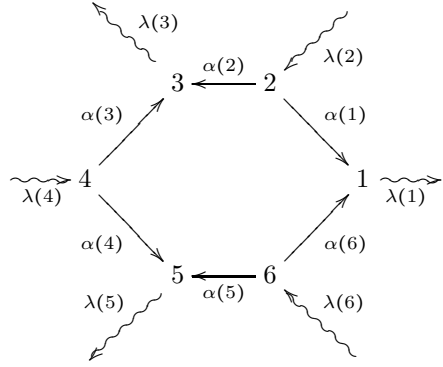
\[
\xymatrix{&&&& \ar@{~>}[dl]^{\lambda(2)} \\
&& 3 \ar@{~>}[ul]_{\lambda(3)} & 2 \ar[l]_{\alpha(2)} \ar[dr]^{\alpha(1)} \\
\ar@{~>}[r]_{\lambda(4)} & 4 \ar[ur]^{\alpha(3)} \ar[dr]_{\alpha(4)} &&& 1  \ar@{~>}[r]_{\lambda(1)}& \\
&& 5 \ar@{~>}[dl]_{\lambda(5)} & 6 \ar[l]^{\alpha(5)} \ar[ur]_{\alpha(6)} \\
&&&& \ar@{~>}[ul]_{\lambda(6)}}
\]
with $n$ vertices along each flag, denoted here by wavy arrows. Label the $j^{th}$ vertex along the $i^{th}$ flag by $(j,i)$, numbered so that $(n,i)$ denotes each center vertex or simply $i$ when it is understood. For consistency we'll always have $\F(i)$ going into the central vertex if $i$ is even and out if $i$ is odd. 

Define the dimension vector $\beta$ as $\beta(j,i) = j$  for each $1\leq i \leq 2k$ and $1 \leq j \leq n$. We will show that the generalized Littlewood-Richardson coefficient in (\ref{one}) is the dimension of the weight space of semi-invariants for a certain weight for this quiver  and the associated dimension vector $\beta$. We have labeled each flag and central arrow by the sequence we will want to eventually associate to it when we calculate the dimension of this particular weight space of semi-invariants. 
More specifically, we'll associate the weakly decreasing sequence $\lambda(i)$ to flag $\F(i)$ with central arrows $\alpha(i-1), \, \alpha(i)$ both entering vertex $(n,i)$ when $i$ is odd and leaving when $i$ is even, and $\alpha(j)$ will likewise denote the partition associated to this arrow.
Throughout the rest of this paper all results except in section \ref{section-others} will be for $k \geq 2$ and the quiver $Q$  with $2k$ flags, which we call the \bi{sun quiver} or the \bi{$2k$-sun quiver} when we want to emphasize the number of flags, and $\beta$ is the dimension vector defined by $\beta(j,i) = j$.

\begin{lemma}
\label{sun-2}
Let $\sigma \in \Z^{Q_0}$ be a weight for the $2k$-sun quiver, $k \geq 2$. If $\dm \si{\beta}_\sigma \neq 0$, then the weight must satisfy $(-1)^i \sigma(j,i) \geq 0$ for all $1 \leq j \leq n, \, 1 \leq i \leq 2k$. Furthermore,
\[
\dm \si{\beta}_\sigma = \sum c^{\phi(1)}_{\alpha(1), \alpha(2)}  c^{\phi(2)}_{\alpha(2), \alpha(3)} \, \cdots \, c^{\phi(2k)}_{\alpha(2k), \alpha(1)}, 
\]
where the sum ranges over all partitions $\alpha(1), \ldots, \alpha(2k),$ and $\phi(i) = (n^{(-1)^i\sigma(n,i)}, \ldots, 1^{(-1)^i \sigma(1,i)})'$ for $1 \leq i \leq 2k$.
\end{lemma}

\begin{proof}
Define $V_j(i) = \CC^{\beta(j,i)}$ as the vector space assigned to vertex $(j,i)$. A standard calculation using Cauchy's rule (see \cite{Ful97a}, page 121) shows that the affine coordinate ring $\CC[\Rep]$ decomposes as a sum of tensor products of irreducible representations of the general linear groups $\gl(V_j(i))$ with contributions from each of the flags and central arrows. Specifically, if $\F(i)$ is a flag going out of a central vertex, meaning when $i$ is odd, then the $n-1$ arrows of the flag contribute
\[
\dsp \bigoplus_{\phi^1(i), \ldots, \phi^{n-1}(i)} S^{\phi^1(i)} V_1(i)^* \otimes \bigotimes_{j=2}^{n-1} (S^{\phi^{j-1}(i)} V_j(i) \otimes S^{\phi^j(i)} V_j(i)^*) \otimes S^{\phi^{n-1}(i)} V_n(i)
\]
for some partitions $\phi^1(i), \ldots, \phi^{n-1}(i)$. We want to determine when these terms give nonzero semi-invariants of weight $\sigma$.  Because the $j^{th}$ term of $\spl(\beta)$ acts trivially on each $S^{\phi^m(i)} V_k(i)$ whenever $j \neq k$, the terms of $\spl(\beta)$ distribute to the corresponding terms of the sum across the tensor products. 

The term $(S^{\phi^1(i)} V_1(i)^*)^{\spl(V_1(i))} \neq 0$ if and only if $\phi^1(i)$ is of size $w \times \dm V_1(i) = w \times 1$ for some $w \in \Z_{\geq 0}$.
Hence, in this case, the space is one-dimensional and is spanned by a semi-invariant of weight $-w$. Therefore, $(S^{\phi^1(i)} V_1(i)^*)^{\spl(V_1(i))}$ contains a nonzero semi-invariant of weight $\sigma (1,i)$ if and only if $\sigma(1,i) < 0$ and $\phi^1(i)$ is of size $-\sigma(1,i)\times 1 = (1^{-\sigma(1,i)})'$. We know from this that $(S^{\phi^1(i)} V_1(i)^*)^{\spl(V_1(i))}$ is nonzero if and only if it is one-dimensional.

Next, $(S^{\phi^1(i)} V_2 (i) \otimes S^{\phi^2(i)} V_2(i)^*)^{\spl(V_2(i))}$ is nonzero if and only if $\phi^2(i)_p - \phi^1(i)_p = k \in \Z_{\geq 0}$ for all $p$. That is, $\phi^2(i)$ is $\phi^1(i)$ plus some extra columns, which must be of length $\dm V_2(i) =\beta(i,2) =  2$. In this case, the space being nonzero is equivalent to it being spanned by a semi-invariant of weight equal to the negative of the number of extra columns. Hence, $(S^{\phi^1(i)} V_2 (i) \otimes S^{\phi^2(i)} V_2(i)^*)^{\spl(V_2(i))}$ contains a semi-invariant of weight $\sigma(2,i)$ if and only if the space is one-dimensional and $\phi^2(i)=(2^{-\sigma(2,i)}, 1^{-\sigma(1,i)})'.$

Reasoning this way and continuing by sorting the spaces of semi-invariants in $\si{\beta}$ of weight $\sigma$, we have that $\phi^1(i)$ is of size $-\sigma(1,i) \times 1$ and $\phi^j(i)$ is attained from $\phi^{j-1}(i)$ by adjoining a rectangle of size $-\sigma(j,i) \times j$ to the left of it. Thus, $\phi^{n-1}(i) = ((n-1)^{-\sigma(n-1,i)}, \ldots, 1^{-\sigma(1,i)})'$ and the contribution of the flag $\F(i)$ for odd $i$ to $\si{\beta}_{\sigma}$ is precisely $S^{\phi^{n-1}(i)} V_n(i)$. 

Similarly, if $\F(i)$ is a flag going into a central vertex, meaning $i$ is even, then $\sigma(j,i) \geq 0$ for all $1 \leq j \leq n-1$ and the contribution of the flag $\F(i)$ to $\dm \si{\beta}_{\sigma}$ is $\dm S^{\phi^{n-1}(i)} V_n(i)^*$ with 
$$\phi^{n-1}(i) = ((n-1)^{\sigma(n-1,i)}, \ldots, 1^{\sigma(1,i)})'.$$

In addition, the $2k$ central arrows give unspecified partitions $\alpha_i$ with at most $n$ parts each. By taking into account the weights at the central vertices and denoting $V_n(i)$ as simply $V(i)$, we may tensor the contributions from the central vertices to the space of semi-invariants with appropriate powers of the determinant to obtain $\gl$-representations, the dimensions of which will be the Littlewood-Richardson coefficients we want. To be precise, we get 
\[
\begin{array}{rcl}
\dm \left((S^{\phi^{n-1}(2i-1)}V(2i-1) \otimes S^{\alpha_{2i-2}} V(2i-1)^* \otimes S^{\alpha_{2i-1}}V(2i-1)^* \otimes \dt_{V(2i-1)}^{-\sigma(n,2i-1)})^{\gl(V(2i-1))}\right) & = & c^{\phi(2i-1)}_{\alpha_{2i-2}, \alpha_{2i-1}} \\\\
\dm \left((S^{\phi^{n-1}(2i)}V(2i)^* \otimes S^{\alpha_{2i}} V(2i) \otimes S^{\alpha_{2i-1}}V(2i) \otimes \dt_{V(2i)}^{-\sigma(n,2i)})^{\gl(V(2i))} \right) &= &c^{\phi(2i)}_{\alpha_{2i-1}, \alpha_{2i}} 
\end{array}
\]
for each $i=1, \ldots, k$ (recall that we consider the central vertex $(n,2k)$ to be the same as $(n,0)$ and so on). Putting these together, the dimension of $\si{\beta}_{\sigma}$ is as stated.

\end{proof}

For weakly decreasing sequences $\lambda(1), \ldots, \lambda(2k) $ of $n$ integers, define the weight $\sigma_1$ as
\begin{equation}
\label{weight-sigma_1}
\sigma_1(j,i) = \begin{cases}
(-1)^i(\lambda(i)_j - \lambda(i)_{j+1}) & 1 \leq i \leq 2k, \, 1 \leq j \leq n-1 \\
(-1)^i \lambda(i)_n & 1 \leq i \leq 2k, \, j = n. \end{cases}
\end{equation}

The following is immediate from calculating what the $\phi(i)$ are with respect to this weight.

\begin{lemma}
\label{si-saturation}
Let $\lambda(1),\ldots, \lambda(2k)$, $k\geq 2$, be weakly decreasing sequences of $n$ integers. Then for every integer $r \geq 1$, we have 
\[
f(r \lambda(1), \ldots, r\lambda(2k)) =  \sum c^{r\lambda(1)}_{\alpha(1), \alpha(2)}  c^{r\lambda(2)}_{\alpha(2), \alpha(3)} \, \cdots \, c^{r\lambda(2k)}_{\alpha(2k), \alpha(1)} = \dm \si{\beta}_{r\sigma_1}.
\]
In particular, when $k=3$ and $r=1$ the dimension of this weight space of semi-invariants is the multiplicity of the branching rule  in (\ref{one}).
\end{lemma}

\begin{proof}
Because the general case is proven in precisely the same way, assume $r=1$. The proof is then the same as that of Lemma \ref{sun-2} because $\phi(i) = \lambda(i)$ with this weight.
\end{proof}

\begin{remark}
While it is clear that $\lambda(1),\ldots, \lambda(2k)$ must be partitions if $f(\lambda(1),\ldots, \lambda(2k))$ is to be nonzero, this is verified from the conditions for $\sigma$ in Lemma \ref{sun-2} and the description of $\sigma_1$.
\end{remark}

\begin{proof}[Proof of Theorem \ref{saturation-sun}]
By representing the multiplicity as the dimension of a weight space of semi-invariants as in Lemma \ref{si-saturation},  the saturation of this multiplicity immediately follows from Theorem \ref{saturation}.
\end{proof}

\begin{remark}
In this way, we have written the generalized Littlewood-Richardson coefficient $f(\lambda(1), \ldots, \lambda(m))$ as the dimension of a certain weight space of semi-invariants of some quiver. However, we can only express the generalized Littlewood-Richardson coefficient in terms of quiver invariant theory when $m$ is even and at least four. When $m \geq 3$ is odd, this process fails because the first and last flags will be oriented the same direction which would require the central arrow connecting the first and last central vertices to be pointed both directions, an impossibility, while if $m=2$ we would have an oriented cycle.
\end{remark}

\vspace{.2in}

\section{The facets of the cone of effective weights}\label{section-facets}

Recall that for a quiver $Q$ and dimension vector $\beta$, the set of (integral) effective weights is 
\[
\Sigma(Q,\beta)= \{\sigma \in \Z^{Q_0} \mid \si{\beta}_\sigma \neq 0 \}.
\]
If $\sigma \in \R^{Q_0}$ is a real-valued function on the set of vertices $Q_0$ and $\alpha$ is an integer-valued function on $Q_0$, define $\sigma(\alpha)$ by 
\[
\sigma(\alpha) = \sum_{x \in Q_0} \sigma(x) \alpha(x).
\]

The condition $\sigma(\beta)=0$ is clearly necessary for $\sigma$ to be effective. This is because the action of the one-dimensional torus $\{(t \Id_{\beta(i)})_{i \in Q_0} \mid t \in k \backslash \{0\}\}$ on $\Rep$ is trivial, so if $f$ is a nonzero semi-invariant of weight $\sigma$ and $g_t = (t\Id_{\beta(i)})_{i \in Q_0} \in \gl(\beta)$, then 
\[
g_t \cdot f = t^{\sigma(\beta)} \cdot f,
\]
which implies $\sigma(\beta)=0$. Surprisingly, satisfying a certain set of linear homogenous inequalities is sufficient for a weight to be effective (see Theorems \ref{spanning} and  \ref{facets}).

King \cite{Kin94} gave the following numerical criterion for $\sigma$-(semi-)stability for finite-dimensional algebras based on the Hilbert-Mumford criterion from GIT. (King's criterion differs in sign from our convention, which is why the inequalities in the following theorem go the opposite direction as the ones in his original paper.)

\begin{theorem} Let $Q$ be a quiver, $\beta$ a dimension vector, and $V \in \Rep$. Suppose $\sigma \in \Z^{Q_0}$ is a weight such that $\sigma(V) = 0$. Then
	\begin{enumerate}
		\item $V$ is $\sigma$-semi-stable if and only if $\sigma(\dmm \,{V'}) \leq 0$ for every subrepresentation $V'$ of $V$;
		\item $V$ is $\sigma$-stable if and only if $\sigma(\dmm \, {V'}) < 0$ for every proper nontrivial subrepresentation $V'$ of $V$.
	\end{enumerate}
We call $\beta$ \emph{$\sigma$-(semi)-stable} if there exists a $\sigma$-(semi-)stable representation in $\Rep$.
\end{theorem}

\begin{remark}
Because of this description of the $\sigma$-(semi)-stable representations, there is a full subcategory of $\text{Rep}(Q)$ consisting $\sigma$-(semi)-stable representations. This is an abelian category with simple objects being the $\sigma$-stable representations, and moreover because every representation has finite length, the subcategory is Artinian and Noetherian, so any $\sigma$-semi-stable representation has a Jordan-H\"{o}lder filtration with $\sigma$-stable factors.
\end{remark}

The following result is quite useful for calculations.

\begin{lemma}[Reciprocity Property]\emph{(\cite{DW00a}, Corollary 1)}
\label{reciprocity}
	For any dimension vectors $\alpha, \beta$ and quiver $Q$ without oriented cycles, we have
	\[
	\dm \si{\beta}_{\langle \alpha, \cdot, \rangle} = \dm \si{\alpha}_{-\langle \cdot, \beta \rangle}.
	\]
\end{lemma}

Denote this common value of the dimensions of the weight spaces by $\alpha \circ \beta$. By the saturation of effective weights (Theorem \ref{saturation}) and the reciprocity property, we have
\[
\alpha \circ \beta \neq 0 \Longleftrightarrow r\alpha \circ s \beta \neq 0, \; \forall r,s \geq 1.
\]

Schofield defined useful semi-invariants which span the weight spaces of semi-invariants, and together with his study of general representations the set $\Sigma(Q,\beta)$ may be described in the following way. We use the notation $\alpha \hookrightarrow \beta$ to mean that every $\beta$-dimensional representation has a subrepresentation of dimension $\alpha$.

\begin{theorem}[\cite{DW00a}, Theorem 3]
	\label{spanning}
	
	Let $Q$ be a quiver and $\beta$ a sincere dimension vector, \ie, $\beta(x) \neq 0$ for all $x \in Q_0$. If $\sigma = \langle \alpha, \cdot \rangle \in \Z^{Q_0}$ is a weight with $\alpha \in \Z^{Q_0}$, then the following statements are equivalent:
	\begin{enumerate}
		\item $\dm \si{\beta}_\sigma \neq 0;$
		\item $\sigma(\beta) = 0$ and $\sigma(\beta') \leq 0$ for every $\beta' \hookrightarrow \beta;$
		\item $\alpha$ is a dimension vector, $\sigma(\beta) = 0$, and $\alpha \hookrightarrow \alpha + \beta$.
	\end{enumerate}
\end{theorem}

Some of the necessary and sufficient linear homogeneous inequalities above turn out to be redundant. In order to describe a minimal list among these, we need the following result.

\begin{theorem}[\cite{Sch92}, Theorem 6.1]
\label{Schofield-Schur}
	Let $\beta \in \Z_{\geq 0}^{Q_0}$ be a dimension vector. The following are equivalent:
	\begin{enumerate}
		\item there exists a $\beta$-dimensional representation $V$ such that End$_Q(V,V) \cong \C;$
		\item $\sigma_\beta(\beta') < 0$ for all $\beta' \hookrightarrow \beta$, $\beta' \neq 0,\, \beta$, where $\sigma_\beta = \langle \beta, \cdot \rangle - \langle \cdot, \beta \rangle.$
	\end{enumerate}
	
	Any such $\beta$ is called a \emph{Schur root} and a representation with these properties is a \emph{Schur representation}.
\end{theorem}

With this theorem, we can describe the effective weights by semi-stability conditions. In particular, we will be able to associate the dimension vectors of the subrepresentations with subsets of $\{1, \ldots, n\}$ because of the following theorem.

\begin{theorem}[\cite{DW11}, Corollary 5.2]
\label{facet-decomp}
	Let $Q$ be a quiver without oriented cycles and $N$ vertices,  and $\beta$ a Schur root. Then 
	\begin{enumerate}
	\item $\dm C(Q, \beta) = N-1$, and 
	\item $\sigma \in C(Q,\beta)$ if and only if $\sigma(\beta) = 0$ and $\sigma(\beta_1) \leq 0$ for every decomposition $\beta = c_1 \beta_1 + c_2 \beta_2$ with $\beta_1, \beta_2$ Schur roots, $\beta_1 \circ \beta_2 = 1$, and $c_i = 1$ whenever $\langle \beta_i, \beta_i \rangle < 0$. 
	\end{enumerate}
\end{theorem}

This then allows us to describe the cone of effective weights as 
\[
\cone = \{\sigma \in \HH(\beta) \mid \sigma(\beta') \leq 0 \text{ for all } \beta' \hookrightarrow \beta\},
\]
where $\HH(\beta) = \{ \sigma \in \RR^{Q_0} \mid \sigma(\beta) = 0\}.$

\begin{remark} 
\label{minimal-list}
While we could replace $\beta_1 \circ \beta_2 =1$ with $\beta_1 \circ \beta_2 \neq 0$ in Theorem \hyperref[facet-decomp]{\ref{facet-decomp}(2)}, this would give a longer list of inequalities. The condition $\beta_1 \circ \beta_2 =1$ ensures a complete and irredundant list of necessary and sufficient inequalities.
\end{remark}

A more precise description of the facets of the cone $C(Q,\beta)$ is given by the following.

\begin{definition}
	For a dimension vector $\beta$, define $W_2(Q,\beta)$ as the set of all ordered pairs $(\beta_1, \beta_2)$ such that:
	\begin{enumerate}
		\item $\beta = c_1 \beta_1 + c_2 \beta_2$ for some integers $c_1, c_2 \geq 1$;
		\item $\beta_1$ and $\beta_2$ are Schur roots;
		\item $s_1 \beta_1 \circ s_2 \beta_2 =1$ for all $s_1, s_2 \geq 1$;
		\item $c_i=1$ whenever $\langle \beta_i, \beta_i \rangle < 0$.
	\end{enumerate}
\end{definition}

\begin{theorem}[\cite{DW11}, Theorem 5.1]
\label{facets}
	Let $Q$ be a quiver without oriented cycles and $\beta$ a Schur root. Let $\mathcal{F}(Q,\beta)$ denote the set of all facets of $C(Q,\beta)$. Then the map
	\[
	W_2(Q,\beta) \to \mathcal{F}(Q,\beta)
	\]
	defined by 
	\[
	(\beta_1, \beta_2) \in W_2(Q,\beta) \mapsto C(Q,\beta_1) \bigcap C(Q,\beta_2) = \mathbb{H}(\beta_1) \bigcap C(Q,\beta)
	\]
	is a bijection. Thus, a minimal list of linear homogeneous inequalities defining the cone $C(Q,\beta)$ is obtained by $\sigma(\beta) = 0$ and $\sigma(\beta_1) \leq 0$ for all $(\beta_1, \beta_2) \in W_2(Q,\beta)$.
\end{theorem}

\vspace{.2in}

\section{The facets of the cone of effective weights for the sun quiver}\label{section-facets-quivers}

In order to use the results in the previous section to describe the facets of $C(Q,\beta)$, we'll first show that the dimension vector $\beta$ is Schur and determine conditions on the $\beta_1'$s that can appear in the decompositions. 

\begin{lemma}
	\label{Schur-root}
	The dimension vector $\beta$ for the sun quiver is Schur.
\end{lemma}

\begin{proof}
	The dimension vector $\beta$ is indivisible, meaning the greatest common divisor of its coordinates is one. By a result of Kac
	(\cite{Kac82}, Theorem B(d)),
to show $\beta$ is Schur, it suffices to show that $\beta$ is in the fundamental region of the graph, meaning that the support of $\beta$ is a connected graph and $\tau_i(\beta) \leq 0$ for all vertices $i \in Q_0$, where $e_i$ denotes the dimension vector of the simple representation at vertex $i$ and  $\tau_i(\cdot) := \langle e_i, \cdot \rangle + \langle \cdot, e_i \rangle$. This is immediately checked to hold for all $n \geq 1$.
\end{proof}

\begin{corollary}
\label{dim-cones}
	For the sun quiver $Q$, $\dm C(Q,\beta) = mn-1$.
\end{corollary}

\begin{proof}
This immediately follows from Lemma \ref{Schur-root} and Theorem \ref{facet-decomp}.
\end{proof}

Now consider the following dimension vectors $\beta_1$ for the sun quiver, where $e_{(j,i)}$ denotes the dimension vector of the simple representation at vertex $(j,i)$:
\begin{enumerate}
\item $\beta_1 = e_{(j,i)}$ for a flag $i$ going out of the central vertex, or $\beta_1 = \beta - e_{(j,i)}$ for a flag $i$ going into a central vertex;
\item $\beta_1 \neq \beta$, $\beta_1 \circ (\beta - \beta_1) = 1$, and $\beta_1$ is weakly increasing with jumps of at most one along each of the $m$ flags.
\end{enumerate}

Denote the set of such $\beta_1$ by $\mathcal{D}$. We show in the next lemma that each $\beta_1 \in \mathcal{D}$ defines a facet of $\cone$, which is called a \textit{regular facet} if $\beta_1$ is in the form described in (2), while a facet defined by some $\beta_1$ as described in (1) is called \textit{trivial}. The interpretation of the inequalities arising from the $\beta_1 \in \mathcal{D}$  is given in Remark \ref{facets-interpretation}.

\begin{lemma}
\label{facets-2}
	For the sun quiver $Q$, the regular facets of $C(Q,\beta)$  are of the form
	\[
	\mathbb{H}(\beta_1) \bigcap C(Q,\beta),
	\]
	where $\beta_1$ is weakly increasing with jumps of at most one along the flags, $\beta_1 \neq \beta$, and $\beta_1 \circ (\beta - \beta_1) = 1$.
\end{lemma}

\begin{proof}
	The proof is the same as the one given in \cite{Chi08}, Lemma 5.2, which we include here for completeness. By Theorem \ref{facets}, a facet $\mathcal{F}$ is of the form  
	\[
	\mathbb{H}(\beta_1) \bigcap C(Q,\beta),
	\]
	where $\beta_1, \beta_2$ are Schur roots, $\beta_1 \circ \beta_2 = 1$, and $\beta = c_1\beta_1 + c_2\beta_2$ for some $c_1, c_2 \geq 1$.
	
	Suppose that $\beta_1$ is not of the form as in (1). We'll show that $\beta_1$ is weakly increasing with jumps of at most one along the flags. Denote $c_1\beta_1 = \beta_1'$ and $c_2\beta_2 = \beta_2'$. Because it is clear that $s_1 \beta_1 \circ s_2 \beta_2 \geq \beta_1 \circ \beta_2$ for all $s_1,s_2 \geq 1$, $\beta_1' \circ \beta_2' \neq 0$. By Theorem \ref{spanning}, it follows that any representation of dimension vector $\beta$ has a subrepresentation of dimension vector $\beta_1'$. Choose a $\beta-$dimensional representation which is injective along the flags going into a central vertex and surjective along the flags going out of the central vertex. Then $\beta_1'$ is weakly increasing along the flags going in and has jumps of at most one (from the end of the flag towards the center vertex) along the flags going out, or else the maps couldn't be surjective.
	
	We'll show that $\beta_1'$ is weakly increasing along each flag $\mathcal{F}(i)$ going out of a central vertex. Suppose to the contrary that $\beta_1'(l+1) - \beta_1'(l) < 0$ for some $l \in \{1, \ldots, n-1\}$. Then $\beta_1' - e_l \hookrightarrow \beta'_1$. Moreover, $\beta_1' \circ \beta_2' \neq 0$ is equivalent to $\beta_1'$ being $-\langle \cdot, \beta_2' \rangle-$semi-stable by reciprocity (Theorem \ref{reciprocity}). Thus, $\langle \beta_1' - e_l, \beta_2' \rangle \geq 0$, so $\beta'_2(l) \leq \beta_2'(l-1)$, implying $\beta_1'(l) \geq 1 + \beta_1'(l-1)$. As we previously showed that $\beta_1'$ has jumps of at most one along such a flag, we must have $\beta_1'(l) = 1 + \beta_1'(l-1)$. Thus, $c_1 = 1$ and $e_l \hookrightarrow \beta_1'$. We then have that $\beta_1' = \beta_1$ is a Schur root by assumption, hence is $\sigma_{\beta_1'}$-semistable  by Theorem \ref{Schofield-Schur}, and $e_l, \, \beta_1' - e_l \hookrightarrow \beta_1'$, with $\beta_1' \neq e_l$. Therefore, by the same theorem, $\sigma_{\beta'_1}(e_l)< 0$ and $\sigma_{\beta'_1}(\beta_1'-e_l)< 0,$ which is a contradiction. Thus, $\beta_1'$ must be weakly increasing along the flags going out of a central vertex. By a similar argument, $\beta_1'$ will have jumps of at most one along each flag going in.
	
	Finally, we'll show that $c_1=c_2=1$. Because $\beta_1' = c_1\beta_1$ has jumps of at most one along each flag, we have $0 \leq c_1(\beta_1(l+1,i) - \beta_1(l,i)) \leq 1$ for all $l \in \{1, \ldots, n-1\}$ and $i \in \{1, \ldots, m\}.$ If there are no $l,i$ such that $\beta_1(l+1,i) - \beta_1(l,i) \neq 0$, then $c_1 = 1$, while otherwise there must exist an $i$ such that $\beta_1'(l,i) =1,$ so $c_1=1$. Similarly, $c_2=1$. 
	
	Thus, $\beta = \beta_1 + \beta_2$ with $\beta_1$ weakly increasing of jumps of at most one along the flags.
\end{proof}

\begin{lemma}
\label{inequalities}
Let $\sigma \in \HH(\beta)$ for the sun quiver. Then $\sigma \in \cone$ if and only if the following are true:
	\begin{enumerate}
		\item $(-1)^i \sigma(e_{(j,i)}) \geq 0$ for all $1 \leq j \leq n-1, \, 1 \leq i \leq m$;
		\item $\sigma(\beta_1) \leq 0$ for every $\beta_1 \neq \beta$ weakly increasing with jumps of at most one along the flags and $\beta_1 \circ (\beta - \beta_1) = 1$.
	\end{enumerate}
\end{lemma}

\begin{proof}
	The description of the regular facets in Lemma \ref{facets-2} proves one direction, while if $\sigma \in \cone$, then $\sigma(\beta_1) \leq 0$ for every $\beta_1 \in \mathcal{D}$ by Theorem \ref{spanning}, which  is equivalent to (1) and (2).
\end{proof}

\begin{remark}
\label{facets-interpretation}
	Let $\sigma_1$ be the weight we defined for the sun quiver in equation (\ref{weight-sigma_1}). In particular,
	\[
	\sigma_1(e_{(j,i)}) = (-1)^i(\lambda(i)_j - \lambda(i)_{j+1}), \quad  \, \,  1 \leq i \leq m, \,  1 \leq j \leq n-1.
	\]
	The inequalities arising from a trivial facet of $C(Q,\beta)$, as described in (1) in the preceding lemma, called the \textit{chamber inequalities}, simply state that the sequences $\lambda(i)$ are weakly decreasing sequences of real numbers. Because we will always assume this, we exclude these $\beta_1$ from our considerations. The inequalities arising from (2) in the lemma are called the \textit{regular inequalities}, and the corresponding facet is regular. 
\end{remark}

\vspace{.2in}

\section{Horn-type inequalities}\label{section-Horn}

In this section let $\beta_1$ be a dimension vector which is weakly increasing with jumps of at most one along each of the flags towards the central vertices. Define the following jump sets:
\[
I_i = \{l  \mid \beta_1(l,i) > \beta_1(l-1,i), \, 1 \leq l \leq n\}
\]
with the convention $\beta_1(0,i) =0$ for all $i$. Because $\beta_1$ defines a tuple $I=(I_1, \ldots, I_m)$, we'll commonly denote $\beta_1$ by $\beta_I$. Note that $|I_i| = \beta_I(n,i)$ for each $i$.

Conversely, each tuple $I = (I_1, \ldots, I_m)$ of subsets of $\{1, \ldots, n\}$ defines a dimension vector $\beta_I$ because if 
\[
I_i = \{z_1(i) < \cdots < z_r(i)\},
\]
then $\beta_I(j,i) = j-1$ for all $z_{k-1}(i) \leq j < z_k(i)$ for all $1 \leq k \leq r+1$, with the convention that $z_0(i) = 0$ and $z_{r+1}(i) = n+1$ for all $i$. This means that going towards the center vertex on the $i^{th}$ flag, the dimension at a vertex is 0 until the vertex $z_1(i)$, at which it becomes 1 and continues to be 1 until the vertex $z_2(i)$, at which point it becomes $2$, and so on.

\begin{definition}
	For the dimension vector $\beta$ associated to the sun quiver, define $T(n,m)$ to be the set of all tuples $I=(I_1, \ldots, I_m)$ such that $\beta_I \neq \beta$ (equivalently, $|I_i|<n$ for some $i$), and $\beta_I \circ (\beta - \beta_I) = 1$. 
\end{definition}

We give a description of the set $T(n,m)$ without reference to the sun quiver and only in terms of partitions in Lemma \ref{sun-Horn}.

\begin{proposition}
\label{Horn-type-inequalities}
	Suppose $\lambda(1),\ldots, \lambda(m)$ are weakly decreasing sequences of $n$ real numbers. For the sun quiver $Q$ and dimension vector $\beta$, $\sigma_1 \in C(Q,\beta)$ if and only if $\sum_{i \text{ even}} |\lambda(i)| = \sum_{i \text{ odd}} |\lambda(i)| $ and 
		\[
	\sum_{j \in I_i} \sum_{i \text{ even}} \lambda(i)_j \leq \sum_{j \in I_i} \sum_{i \text{ odd}} \lambda(i)_j
		\]
		for every tuple $(I_1, \ldots, I_m) \in T(n,m)$.
\end{proposition}

\begin{proof}
We prove the statement by using  the definition of the set $T(n,m)$, Theorem \ref{facets}, and Lemma \ref{inequalities}. Calculating directly from the definition of the weight $\sigma_1$ in equation (\ref{weight-sigma_1}) of Section \ref{section-saturation}, 
	\[
	\sigma_1(\beta) = \sum_{i \text{ odd}} \sum_{j=1}^{n-1}[(\lambda(i)_{j+1} - \lambda(i)_j) \beta(j,i)] 
	+ \sum_{i \text{ even}} \sum_{j=1}^{n-1}[(\lambda(i)_j - \lambda(i)_{j+1}) \beta(j,i)]
	+ \sum_{i=1}^m (-1)^i \lambda(i)_n \beta(n,i) 
	\]
	Substituting $\beta(j,i)=j$, $\sigma_1(\beta) = 0$ precisely when the equality  holds. Replacing $\beta$ with $\beta_I$, we get a similar expression, and after noting that $\lambda(i)_j$ contributes to the inequality exactly when $\beta_I(j,i) \neq \beta_I(j-1,i)$, meaning $j \in I_i$, we obtain the inequality.  Because the $\beta_I$ are precisely those described in Lemma \ref{inequalities}, this proves the equivalence. 
\end{proof}

We now want to better understand the set $T(n,m)$. Specifically, we would like to describe the set without any intrinsic reference to $\beta_I$ but rather in terms of partitions. For any tuple $(I_1, \ldots, I_m)$ of subsets of $\{1,\ldots,n\}$, define the following decreasing sequences of integers, where we identify $I_0$ and $I_m$:
\[
\unlam(I_i) = \begin{cases}
\lambda'(I_i) & i \text{ even} \\
\lambda'(I_i) -((|I_i| - |I_{i-1}| - |I_{i+1}|)^{n-|I_i|}) & i \text{ odd}.
\end{cases}
\]

\begin{lemma}
\label{sun-Horn} 
The set $T(n,m)$ for the sun quiver consists of all tuples $I = (I_1, \ldots, I_m)$ such that: 
\begin{enumerate}
	\item[(a)] at least one of the subsets $I_1, \ldots, I_m$ has cardinality $< n$;
	\item[(b)] $\unlam(I_i)$ is a partition for all $1 \leq i \leq m$;
	\item[(c)] $f(\unlam(I_1), \ldots, \unlam(I_m)) = 1$.
\end{enumerate}
\end{lemma}

\begin{proof} Denote the weight $\langle \beta_1, \cdot \rangle$ by $\sigma_I$. Describing $\beta_1$ by $\beta_I$ as previously and letting $e(j,i)$ be the dimension vector of the simple representation with support at vertex $j$ on the $i^{th}$ flag, the contribution to $\sigma_I(\alpha)$ at a vertex $l \in \{1, \ldots, n-1\}$ on a flag $i$ is
\[
\langle \beta_I, e(l,i) \rangle = \beta_I(l)  - \beta_I(l-1)
\]
if $i$ is even with $\beta_I(0) = 0$, and 
\[
\langle \beta_I, e(l,i) \rangle = \beta_I(l)  - \beta_I(l+1) 
\]
if $i$ is odd. Since $\beta_I$ is weakly increasing with jumps of at most one along the flags, this translates to 
\begin{equation}
	\label{eq:weight-even}
\sigma_I(l,i) = \left\{\begin{array}{lr} 
1 & \text{ if }l \in I_i \\
0 & \text{otherwise}
\end{array} \right.
\end{equation}
for $l \in \{1, \ldots, n-1\}$ and $i$ even, and 
\begin{equation}
	\label{eq:weight-odd}
\sigma_I(l,i) = \left\{ \begin{array}{lr}
-1 & \text{ if } l+1 \in I_i \\
0 & \text{otherwise}
\end{array} \right.
\end{equation}
for $l \in \{1, \ldots, n-1\}$ and $i$ odd. We then only need to describe the weight at the central vertices. We have
\[
\sigma_1(n,i) = \begin{cases}
\beta_I(n,i) - \beta_I(n-1,i) & i \text{ even} \\
\beta_I(n,i) - \beta_I(n,i-1) - \beta_I(n,i+1) & i \text{ odd},
\end{cases}
\]
where $\beta_I(n,0) = \beta_I(n,m)$. We use the fact that $\beta_I(n,i) = |I_i|$ to finish defining the weight $\sigma_I$ in terms of the subsets $I_i$, where we identify $I_0$ as $I_m$ below:

\begin{equation}
\sigma_I(n,i) = \left\{\begin{array}{lr}
0 &  i  \text{ even}, \; n \not \in I_i, \\
1 & i  \text{ even}, \; n \in I_i,\\
|I_i| - |I_{i-1}| - |I_{i+1}| & i \text{ odd}.
\end{array}
\right.
\end{equation}

Using this explicit description of this weight $\sigma_I$, we can calculate the corresponding partitions in determining $\dm \si{\beta}_{\sigma_I}$. The calculation is done precisely the same as before (see Lemma  \ref{sun-2} for the details). The partitions we want to consider for describing the contribution of the first $n-1$ vertices along each flag are 
\[
\gamma_{n-1}(i) = (\beta_2(n-1,i)^{(-1)^i\sigma_I(n-1,i)}, \ldots, \beta_2(1,i)^{(-1)^i \sigma_I(1,i)})', \qquad 1 \leq i \leq m.
\]
Recalling the contributions of the central vertices to the space of semi-invariants, we have that
\[
\gamma(i) = \gamma_{n-1}(i) + (((-1)^i \sigma_I(n,i))^{\beta_2(n,i)}), 
\]
are the partitions we want for $1 \leq i \leq m$, and 
\[
\dm \ssi(Q,\beta_2)_{\sigma_1} = f(\gamma(1),\ldots,\gamma(m)).
\]

Since $\beta_2 = \beta - \beta_1$, if $I_i = \{z_1(i) < \ldots < z_r(i)\}$, then 
\[
\beta_2(z_j(i),i) = z_j(i) - j = z_j(i) - 1 - (j-1) = \beta_2(z_j(i)-1,i),
\]
and in particular, $\beta_2(n,i) = n - \beta_1(n,i) = n - |I_i|$. With this, $\unlam(I_i) = \gamma(i)$ for each $i$.

Thus, if $I = (I_1, \ldots, I_m) \in T(n,m)$, then condition (a) is satisfied by definition, (b) is satisfied since $\gamma(i)$ is a partition for each $i$ and $\gamma(i) = \unlam(I_i)$, and (c) is true because
\[
1 = \beta_I \circ (\beta - \beta_I) = f(\gamma(1), \ldots, \gamma(m)) = f(\unlam(I_1), \ldots, \unlam(I_m)).
\]
Conversely, if conditions (a)$ - $(c) are satisfied by a tuple $I= (I_1, \ldots, I_m)$ of subsets of $\{1,\ldots, n\}$, then we can construct a dimension vector $\beta_I$ in the usual way and the associated decreasing sequences of integers $\unlam(I_i), \, 1 \leq i \leq m$. Necessarily, $\beta_I \neq \beta$ and 
\[
\beta_I \circ (\beta - \beta_I) = f(\unlam(I_1), \ldots, \unlam(I_m)) = 1
\]
Thus, $I \in T(n,m)$.
\end{proof}

\begin{proposition}
\label{conditions-weight-cone}
Let $\lambda(1), \ldots, \lambda(m)$ be weakly decreasing sequences of $n$ reals, $m \geq 4$ and even. The following are equivalent for the sun quiver $Q$:
	\begin{enumerate}
		\item $\sigma \in \cone$;
		\item the numbers $\lambda(i)_j$ satisfy 
		\[
		\sum_{i \text{ even}} |\lambda(i)| = \sum_{i \text{ odd}} |\lambda(i)|\]
		 and 
	\[
		\sum_{j \in I_i} \sum_{i \text{ even}} \lambda(i)_j \leq \sum_{j \in I_i} \sum_{i \text{ odd}} \lambda(i)_j
	\]
		for every tuple $(I_1, \ldots, I_m)$ for which $|I_i| < n$ for some $i$, the $\unlam(I_i)$ are partitions, $1 \leq i \leq m$, and 
		\[
		f(\unlam(I_1), \ldots, \unlam (I_m)) \neq 0;
		\]
		\item the numbers $\lambda(i)_j$ satisfy the same conditions as in part (2), and it suffices to check only the tuples $(I_1,\ldots, I_m)$ which satisfy those conditions and
		\[
		f(\unlam(I_1), \ldots, \unlam(I_m)) = 1.
		\]
	\end{enumerate}
\end{proposition}

\begin{proof}
This follows from Proposition \ref{Horn-type-inequalities}, Lemma \ref{sun-Horn}, and Remark \ref{minimal-list}.
\end{proof}

We can deduce minor conditions on the sizes of the $I_i$.

\begin{lemma}
	Let $I = (I_1,\ldots, I_m)$ be a tuple of subsets of $\{1, \ldots, n\}$ and define $s_i$ to be the smallest $k \in \{0, \ldots, |I_i|\}$ such that $n - k \notin I_i$. Then
	\[
	\max \{|I_{i-1}|, |I_{i+1}| \} \leq |I_i| \leq |I_{i-1}| + |I_{i+1}| + s_i
	\]
	if $I \in T(n,m)$ and $i$ is odd.
\end{lemma}

\begin{proof}
	If $\beta_1 \circ \beta_2 \neq 0$, then any representation $V$ of dimension $\beta_1 + \beta_2$ has a subrepresentation of dimension $\beta_1$. Choosing $V$ such that $V(a)$ is invertible for every arrow $a$ between central vertices, we immediately have 
	\[
	\text{max} \{|I_{i-1}|, |I_{i+1}| \} \leq |I_i|
	\]
	for the stated $i$. Let $\langle \beta_1, \cdot \rangle = \sigma_1$. In order for $\dm \ssi (Q,\beta_2)_{\sigma_1}$ to be nonzero, each $\unlam(I_i)$ must be a partition, meaning, in particular, that it has non-negative parts. Note that $s_i$ is precisely the smallest part of $\lambda'(I_i)$ for each $i$. We have $\unlam(I_i) = \lambda'(I_i) - (\sigma_1(i)^{n-|I_i|})$ for each of the specified $i$, as well as $\sigma_1(i) = |I_i| - |I_{i-1}| - |I_{i+1}|$. Thus, $\unlam(I_i)$ is a partition if and only if 
	\[
	0 \leq |I_{i-1}| + |I_{i+1}| - |I_i| + s_i.
	\]
\end{proof}

We note that the proof above does not extend to determining how the sizes of the other subsets compare or whether some subset contains $n$ or is nonempty. This is because no central vertex in the sun quiver shares two flags, so  we can't cancel the sizes of any two consecutive subsets that appear in $\sigma_1$.

\begin{example}
\label{n=2}
For the case $n=2$ and $m=6$ for the sun quiver, we can compute all decompositions $\beta = \beta_1 + \beta_2$ such that both $\beta_1, \beta_2$ are Schur roots and $\beta_1 \circ \beta_2 = 1$, or equivalently by Theorem \ref{facets}, a description of all facets of $\cone$ in this case. This is done by first computing all tuples $(\unlam(I_1),\ldots, \unlam(I_6))$ such that $f(\unlam(I_1),\ldots, \unlam(I_6))=1$ and such that the other conditions of Proposition \ref{conditions-weight-cone} are satisfied. Next, we compute all tuples $(I_1,\ldots, I_6)$ arising from these and omit any such that the corresponding $\beta_1$ and $\beta_2$ are not Schur as these will produce redundant inequalities. The inequalities are immediately determined from these tuples $(I_1,\ldots, I_6)$. We list  the corresponding dimension vectors $\beta_1$ in the \hyperref[appendix]{Appendix}.

The corresponding inequalities of the partitions arising from these decompositions then provide a complete and minimal list of linear homogeneous inequalities defining when a weight $\sigma$ is in $\cone$ for the sun quiver $Q$ in this case. Specifically, $\sigma \in \cone$ if and only if the defining partitions satisfy 
\[
|\lambda(1)| + |\lambda(3)| + |\lambda(5)| = |\lambda(2)| + |\lambda(4)| + |\lambda(6)|,
\]
and 
\[
\begin{array}{ll}
\lambda(2)_1 \leq \lam(1)_1 + \lam(3)_1 & 
 \lam(2)_2 \leq \lam(1)_1 + \lam(3)_2\\\\
 \lam(2)_1 + \lam(4)_2 \leq \lam(1)_1 + \lam(3)_1  + \lam(5)_1 & 
 \lam(2)_2 + \lam(4)_2 \leq \lam(1)_1 + \lam(3)_2  + \lam(5)_1\\\\
 \lam(2)_2 + \lam(6)_2 \leq \lam(1)_1 +  \lam(3)_2 + \lam(5)_1  &
 |\lam(2)| \leq |\lam(1)| + |\lam(3)| \\\\
 \lam(2)_1 + \lam(4)_1  \leq \lam(1)_1  +  |\lam(3)| + \lam(5)_1 & 
 \lam(2)_1 + \lam(4)_2  \leq \lam(1)_1 + |\lam(3)| + \lam(5)_2 \\\\
 \lam(2)_2 + \lam(4)_2 \leq \lam(1)_2 + |\lam(3)|  + \lam(5)_2 & 
\lam(2)_1 + \lam(4)_2 + \lam(6)_2 \leq \lam(1)_1 + |\lam(3)|  + \lam(5)_1  \\\\
\lam(2)_1 + \lam(4)_2 + \lam(6)_2 \leq \lam(1)_1 +  \lam(3)_1  +|\lam(5)|  &
 \lam(2)_2 +  \lam(4)_2  + \lam(6)_2  \leq \lam(1)_1 + |\lam(3)|  + \lam(5)_2 \\\\
 |\lam(2)| + \lam(4)_1 + \lam(6)_2 \leq |\lam(1)| + |\lam(3)|  + \lam(5)_1  & 
 |\lam(2)|+ \lam(4)_2 + \lam(6)_2  \leq |\lam(1)| + |\lam(3)|  + \lam(5)_2 ,
\end{array}
\]
along with the inequalities obtained by permutations of the indices that respect the symmetries of the sun quiver. This likewise provides a description of all the $(I_1, \ldots, I_6)$ in $T(2,6)$, as described in Proposition \ref{Horn-type-inequalities}.

\end{example}

\vspace{.2in}

 \section{Generalized eigenvalue problem}\label{Hermitian}

The original motivation in \cite{DW00a} for describing Littlewood-Richardson coefficients in terms of quivers was to provide a solution to a famous conjecture of Horn \cite{Hor62}.

\begin{theorem}[Horn's conjecture]\label{Horn}
Let $\lambda(i) = (\lambda_1(i),\ldots, \lambda_n(i))$, $i \in \{1,2,3\}$, be weakly decreasing sequences of $n$ real numbers. Then the following are equivalent:
\begin{enumerate}
\item there exist $n \times n$ complex Hermitian matrices $H(i)$ with eigenvalues $\lambda(i)$ such that 
\[
H(2) = H(1) + H(3);
\]
\item the numbers $\lambda_j(i)$ satisfy 
\[
|\lambda(2)| = |\lambda(1)| + |\lambda(3)|
\]
together with 
\[
\sum_{j \in I_2} \lambda_j(2) \leq \sum_{j \in I_1} \lambda_j(1) + \sum_{j \in I_3} \lambda_j(3)
\]
for every triple $(I_1,I_2,I_3)$ of subsets of $\{1,\ldots, n\}$ of the same cardinality $r < n$ and $c^{\lambda(I_2)}_{\lambda(I_1),\lambda(I_3)} \neq 0;$
\item if $\lambda_j(i)$ is an integer for each $1 \leq j \leq n, \, i \in \{1,2,3\}$, (1) and (2) are equivalent to $c^{\lambda(2)}_{\lambda(1),\lambda(3)} \neq 0$.
\end{enumerate}
\end{theorem}

Klein \cite{Kle68} noted the equivalence of (3) with short exact sequences of finite abelian $p$-groups, while Klyachko \cite{Kly98} proved the equivalence of (1) and (2) in Horn's conjecture. In the same paper, Klyachko noted the connection between this problem and Littlewood-Richardson coefficients. A key step in the proof is the saturation of Littlewood-Richardson coefficients (Theorem \ref{LR-saturation}), which was first proved combinatorially by Knutson and Tao \cite{KT99} and later by Derksen and Weyman \cite{DW00a} in the context of quivers by using the saturation of weight spaces of semi-invariants.

Klyachko found that the set of solutions in part (1) in the statement of Horn's conjecture forms a rational convex polyhedral cone $K(n,3)$ in $\RR^{3n}$, now known as \textit{Klyachko's cone}. In \cite{KTW04}, Knutson, Tao, and Woodward used honeycombs to describe the facets of Klyachko's cone. They found that $K(n,3)$ consists of triples $(\lambda(1), \lambda(2), \lambda(3))$ satisfying the conditions in part (2), and restricting to the triples such that $c_{\lambda(I_1), \, \lambda(I_3)}^{\lambda(I_2)} =1$ provides a minimal list.

As we've previously proven corresponding statements for parts (2) and (3) for the generalized Littlewood-Richardson coefficient $f$, we now want to find the Weyl-type eigenvalue problem for the non-vanishing of this multiplicity.

\subsection{Necessary lemmas}

Before we state the generalized eigenvalue problem for the sun quiver, we state two results from linear algebra that will be fundamental for us. 

\begin{proposition}[\cite{Chi08}, Proposition 7.1]
\label{matrix-equations}
	Let $Q$ be a quiver without oriented cycles, $\beta$ a dimension vector, and $\sigma \in \RR^{Q_0}$. The following are equivalent:
	\begin{enumerate}
		\item $\sigma \in \cone$;
		\item there exists $W = \{W(a)\}_{a \in Q_1} \in \Rep$ satisfying
		\begin{equation*}
			\sum_{\substack{a \in Q_1 \\ ta = x}} W(a)^*W(a) - \sum_{\substack{a \in Q_1 \\ ha = x}} W(a)W(a)^* = \sigma(x) \Id_{\beta(x)}
		\end{equation*}
		for all $x \in Q_0$, where $W(a)^*$ is the adjoint of $W(a)$ with respect to the standard Hermitian inner product on $\C^n$.
	\end{enumerate}
\end{proposition}

\begin{lemma}
\label{Hermitians}
	Let $\sigma(1), \ldots, \sigma(n-1)$ be non-positive real numbers. The following are equivalent:
		\begin{enumerate}
			\item there exist $W_i \in \Mat_{i \times (i+1)}(\C), \, 1 \leq i \leq n-1,$ such that 
			\[
			\begin{array}{rcl}
			W_i W_i^* - W_{i-1}^* W_{i-1} & =&  - \sigma(i) \Id_{i} \quad \text{ for } 2 \leq i \leq n-1, \\
			W_1 W_1^* & = & -\sigma(1);
			\end{array}
			\]
			\item the matrix $H=W_{n-1}^*W_{n-1}$ is Hermitian and has eigenvalues
			\[
			v(k) = - \sum_{j=k}^{n-1} \sigma(j), \quad \forall \, 1 \leq k \leq n-1 \quad \text{and} \quad v(n) = 0.
			\]
		\end{enumerate}
\end{lemma}

\begin{proof}
See Section 3.4 of \cite{CG02}.
\end{proof}

This result will allow us to build up an $n \times n$ Hermitian matrix along each flag, though the result as stated only applies to flags that are going out from a central vertex. For the flags going into a central vertex, we need to use the dual form of the above lemma. Namely, $\sigma(i)$ is a non-negative number for each $1 \leq i \leq n-1$, so replace each $-\sigma(i)$ above with $\sigma(i)$, and switch the order of multiplication of $W_i$ and $W_i^*$ in each case for $W_i \in \Mat_{(i+1)\times i}(\C)$.

\subsection{Generalized eigenvalue problem for $f$}

Recall the construction of the $2k$-sun quiver in Section \ref{section-saturation}. The weight for this quiver is

\[
\sigma_1(j,i) = \begin{cases}
(-1)^i(\lambda(i)_j - \lambda(i)_{j+1}) & 1 \leq i \leq 2k, \, 1 \leq j \leq n-1 \\
(-1)^i \lambda(i)_n & 1 \leq i \leq 2k, \, j = n. \end{cases}
\]

From Proposition \ref{matrix-equations}, $\sigma_1 \in \cone$ if and only if there is a representation $W \in \Rep$ satisfying the specified matrix equations. These equations are essentially the same as those in Lemma \ref{Hermitians} for a flag $\F(i)$ going out of a central vertex, meaning when $i$ is odd, or in the dual statement when $\F(i)$ is going into a central vertex, meaning when $i$ is even. In either case, the first $n-1$ vertices provide $n \times n$ Hermitian matrices $H(i)$ with eigenvalues 
\[
(\lambda(i)_1 - \lambda(i)_n, \ldots, \lambda(i)_{n-1} - \lambda(i)_n, 0)
\]
for each $1 \leq i \leq m$. 

We now consider the equations arising from the central vertices. Denote the $(n-1)^{th}$ arrow along the $i^{th}$ flag as simply $b^i$ and denote the arrows between the central vertices by the usual partition labeling. The equations arising from the central vertices are, for $1 \leq i \leq k$, 
\[
\begin{array}{rcl}
W(b^{2i-1})^*W(b^{2i-1}) - W(\alpha(2i-1))W(\alpha(2i-1))^* - W(\alpha(2i-2))W(\alpha(2i-2))^* &=& -\lambda(2i-1)_n\Id_n, \\\\
W(\alpha(2i-1))^*W(\alpha(2i-1)) + W(\alpha(2i))^*W(\alpha(2i)) - W(b^{2i})W(b^{2i})^* &=& \lambda(2i)_n\Id_n,
\end{array}
\]
where $\alpha(0) = \alpha(m)$. 
We may rewrite these equations by making a few simple observations. Lemma \ref{Hermitians} gives the Hermitian matrices as $H(i) = W(b^i)^*W(b^i)$, or $W(b^i)W(b^i)^*$ depending on the direction of the flag. Clearly, since each $H(i)$ is Hermitian with spectrum $(\lambda(i)_1 - \lambda(i)_n, \ldots, \lambda(i)_{n-1} - \lambda(i)_n, 0)$, $H(i) + \lambda(i)_n \Id_n$ is Hermitian with spectrum $\lambda(i)$; denote this new Hermitian matrix again by $H(i)$. We may conjugate the equations by unitary matrices, if necessary. Moreover, for any $n \times n$ matrix $A$, both $AA^*$ and $A^*A$ are positive semi-definite and have the same spectra, and any positive semi-definite Hermitian matrix $B$ can be written as $WW^*$ or $W^*W$, 
so we may simplify the forms of the equations. We conclude that $\sigma_1 \in \cone$ if and only if there are Hermitian matrices $H(i)$ with spectra $\lambda(i), \, 1 \leq i \leq m,$ and positive semi-definite $n \times n$ matrices $B(\alpha(i))$ such that 
\begin{equation}
\label{Horn-Sun}
H(i) = B(\alpha(i)) + B(\alpha(i-1)), \quad 1 \leq i \leq m, \;\;  
\end{equation}
where $B(\alpha(0)) = B(\alpha(m))$.
Solving for any of the $B(\alpha(i))$ gives $\sum_{i \text{ even}} H(i) = \sum_{i \text{ odd}} H(i)$. Furthermore, because each $H(i)$ is a sum of positive semi-definite matrices, each $H(i)$ must have non-negative eigenvalues. In addition, we get several other conditions on the Hermitian matrices, namely, we can express alternating sums of an odd number of consecutive indexed matrices as a sum of positive semi-definite matrices. Specifically,
\[
H(i) - H(i+1) + \cdots - H(i+j-1) + H(i+j) = B(i-1) + B(i+j), \quad j \in \{0,2,4,\ldots, 2k-2\},
\]
where we are taking $H(m+1) = H(1)$, and so on in cyclic fashion. Thus, each such alternating sum is positive semi-definite. (There is, of course, some redundancy in this statement and the previously stated conditions on the $H(i)$.) These are all the conditions on the $H(i)$ which we can conclude from (\ref{Horn-Sun}). Thus, we've found the necessary conditions, stated above, posing the following problem and proving the subsequent statement.\smallskip

\noindent \textbf{Generalized eigenvalue problem for $f$}.
For which weakly decreasing sequences $\lambda(1),\ldots, \lambda(2k)$, $k \geq 2$, of $n$ non-negative real numbers do there exist $n \times n$ complex Hermitian matrices $H(1), \ldots, H(2k)$ with eigenvalues $\lambda(1),\ldots, \lambda(2k)$ such that 
\[
\sum_{i \text{ even}} H(i) = \sum_{i \text{ odd}} H(i), 
\]
and such that
\[
H(i) - H(i+1) + \cdots - H(i+j-1) + H(i+j), \quad j \in \{0,2, \ldots, 2k-2\},
\]
has non-negative eigenvalues, where $H(2k+1) = H(1)$ and so on in cyclic fashion?

\begin{proposition}
\label{eigenvalue-sun}
Suppose $\lambda(1),\ldots, \lambda(2k)$, $k \geq 2$, are weakly decreasing sequences of $n$ non-negative real numbers, and let $Q$ be the sun quiver, $\beta$ the standard dimension vector, and $\sigma_1$ the weight defined in equation (\ref{weight-sigma_1}). If $\sigma_1 \in \cone$, then there exist $n \times n$ complex Hermitian matrices $H(1), \ldots, H(2k)$ with eigenvalues $\lambda(1),\ldots, \lambda(2k)$ that solve the generalized eigenvalue problem for the multiplicity $f$.
\end{proposition}

While an effective weight defines Hermitian matrices satisfying these conditions, the conditions on the matrices are not sufficient; counterexamples are easily found. Alone, they do not determine a weight because we cannot recapture the decompositions of each $H(i)$ into a sum of the particular positive semi-definite matrices, no canonical choice being available. Any additional conditions would need to record the ``linkage" between the consecutive $H(i)$, that is, the fact that they share a common positive semi-definite matrix in their decompositions.

Define the set $K(n,m) \subseteq \R^{mn}$, $m \geq 4$ and even,  to be all $m$-tuples $(\lambda(1),\ldots, \lambda(m))$ of weakly decreasing sequences of $n$ reals that satisfy $\sum_{i \text{ even}} |\lambda(i)| = \sum_{i \text{ odd}} |\lambda(i)|$ and 
		\begin{equation*}
		\sum_{j \in I_i} \sum_{i \text{ even}} \lambda(i)_j \leq \sum_{j \in I_i} \sum_{i \text{ odd}} \lambda(i)_j
		\end{equation*}
		for every tuple $(I_1, \ldots, I_m)$ such that the $\unlam(I_i)$, $1 \leq i \leq m$, are partitions and 
		\[
		f(\unlam(I_1), \ldots, \unlam (I_m)) \neq 0.
		\]
This makes $K(n,m)$ a rational convex polyhedral cone in $\R^{mn}$, which we call the \bi{generalized Klyachko's cone} for this eigenvalue problem.

\begin{proof}[Proof of Theorem \ref{generalization-Horn}]
The first and third statements follow from Proposition \ref{conditions-weight-cone}, while the second follows from additionally Proposition \ref{eigenvalue-sun}. 
Letting $Q$ denote the sun quiver, there is a map of cones
\[
  K(n,m)  \to  C(Q, \beta) \qquad \quad 
  (\lambda(1), \ldots, \lambda(m))  \mapsto \sigma_1.
\]
This map is an isomorphism of cones by the chamber inequalities in Lemma \hyperref[inequalities]{\ref{inequalities}(1)} and Proposition \ref{eigenvalue-sun}. We found the dimension of $C(Q, \beta)$  to be $mn-1$ in Corollary \ref{dim-cones}, which proves the last statement.
\end{proof}

\vspace{.2in}

\section{Factorization formula}\label{section-factorization}

Derksen and Weyman showed the following result for the star quiver (the quiver they used to represent a single Littlewood-Richardson coefficient as the dimension of a weight space of semi-invariants), where $\beta$ is the corresponding dimension vector.

\begin{theorem}[\cite{DW11}, Theorem 7.8]
\label{decomp}
	For all $\beta_1, \beta_2$ such that $\beta = \beta_1 + \beta_2$, where $\beta_1, \beta_2$ are nondecreasing along the flags and $\beta_1 \circ \beta_2 =1$, the inequality $\sigma(\beta_1) \leq 0$ defines a wall of $\RR^+ C(Q, \beta)$. Furthermore, all nontrivial walls are obtained in this way. 
\end{theorem}

The same proof for the first part of the theorem applies to the sun quiver because the proof only relies on $Q$ being acyclic, $\beta$ being Schur, and the $\beta_1, \beta_2$ satisfying the given assumptions. The second part of the statement also easily extends to our setup. 
Thus, the each weight $\sigma_1$ satisfying the previously defined conditions  defines a wall of $\RR^+ C(Q_\beta)$, and all nontrivial walls of the cone are defined in this way. Because of this, we can extend their proof of the factorization formula for Littlewood-Richardson coefficients to generalized ones.

\begin{definition}
Let $Q$ be a quiver without oriented cycles, $\alpha$ a dimension vector, and $\sigma$ a weight such that $\sigma(\alpha) = 0$. We call $\alpha$ {\it $\sigma$-(semi-)stable} if a general representation of dimension $\alpha$ is $\sigma$-(semi-)stable. We write a decomposition of $\alpha$ into smaller dimension vectors as $\alpha = \alpha_1 \dotplus \ldots \dotplus \alpha_s$ and call this the {\it $\sigma$-stable decomposition} of $\alpha$ if a general representation $V$ of dimension $\alpha$ has a Jordan-H\"{o}lder filtration with composition factors of dimension $\alpha_1, \ldots, \alpha_s$, in some order, including multiplicity.
\end{definition}

We may rewrite the $\sigma$-stable decomposition of a dimension vector $\alpha$ by grouping together the common sub-dimension vectors. If $\alpha_i$ occurs $c_i$ times as the dimension vector of a composition factor in the Jordan-H\"{o}lder filtration of $\alpha$, we write the $\sigma$-stable decomposition of $\alpha$ as 
\[
\alpha = c_1 \cdot \alpha_1 \dotplus c_2 \cdot \alpha_2 \dotplus \ldots \dotplus c_s \cdot \alpha_s,
\]
where $c_i \in \Z^+$ for all $i$ and $\alpha_i \neq \alpha_j$ if $i \neq j$.

\begin{theorem}[\cite{DW11}, Theorem 3.20]
\label{si-factorization}
	Suppose $\sigma$ is an indivisible weight. If 
	\[
	\alpha = c_1 \cdot \alpha_1 \dotplus c_2 \cdot \alpha_2 \cdot \dotplus \ldots \dotplus c_r \cdot \alpha_r
	\]
	is the $\sigma$-stable decomposition of $\alpha,$ then for any $s \in \Z$, there is an equality
	\[
	\dm \ssi(Q,\alpha)_{s \sigma} = \prod_{i=1}^r \dm (S^{c_i} (\ssi(Q, \alpha_i)_{s\sigma})),
	\]
	where $S^{c_i}$ is the $c_i^{th}$-symmetric power.
\end{theorem}

Theorem \ref{factorization} rephrases Theorem 7.14 in \cite{DW11} for generalized Littlewood-Richardson coefficients, and the proof is essentially the same.

\begin{proof}[Proof of Theorem \ref{factorization}]
The conditions on the subsets $I_j$ are precisely those defining the set $T(n,m)$ as shown in Lemma \ref{sun-Horn}. This set describes when the multiplicity $f$ is nonzero, as shown in 
 Proposition \ref{Horn-type-inequalities}. Furthermore, because the multiplicity $f$  agrees with the dimension of the respective weight space of semi-invariants, we can use Theorem \ref{decomp}.

	 If $\sigma_1$ is in the interior of the wall, then the $\sigma_1$-stable decomposition of $\beta$ is $\beta_1 \dotplus \beta_2$. The weight $\sigma_1$ is indivisible, so we may use Theorem \ref{si-factorization}  to get
	\[
	f(\lambda(1),\ldots, \lambda(m))= \dm \ssi(Q,\beta)_{\sigma_1} = \alpha \circ \beta = (\alpha \circ \beta_1)(\alpha \circ \beta_2) = f (\lambda(1)^*,\ldots, \lambda(m)^*) \cdot f (\lambda(1)^\#,\ldots, \lambda(m)^\#).
	\]
	If, on the other hand, $\sigma_1$ is not in the interior of a wall, then the $\sigma_1$-stable decompositions of $\beta_1$ and $\beta_2$ are of the form
	\[
	\beta_1 = c_1\cdot \gamma_1 \dotplus \ldots \dotplus c_s \cdot \gamma_s, \qquad \beta_2 = d_1 \cdot \delta_1 \dotplus \ldots \dotplus d_t \cdot \delta_t.
	\]
	Thus, the $\sigma_1$-stable decomposition of $\beta$ is the sum of these. Because the sets $\{\gamma_1, \ldots, \gamma_s\}$ and $\{\delta_1, \ldots, \delta_t\}$ are disjoint and $\gamma_i \circ \delta_j = 1$ for all $i,j$, Theorem \ref{si-factorization} again gives
	\[
	\begin{array}{rcl}
	f(\lambda(1),\ldots, \lambda(m)) & =&  \alpha \circ \beta \\
	& = & \prod (\alpha \circ (c_i \cdot \gamma_i)) \prod (\alpha \circ (d_i \cdot \delta_i)) \\
	& = & (\alpha \circ \beta_1)(\alpha \circ \beta_2) \\
	& = &  f (\lambda(1)^*,\ldots, \lambda(m)^*) \cdot f (\lambda(1)^\#,\ldots, \lambda(m)^\#).
	\end{array}
	\]
\end{proof}

\vspace{.2in}

\section{Level-1 weights and stretched polynomials}\label{stretched-polynomials}

The stretched function $f(N) = c^{N\nu}_{N\lambda, N\mu}$ for $N \in \Z^+$ for fixed partitions $\lambda, \mu, \nu$ has interesting combinatorial properties and has been studied by many people (see, for instance, \cite{KTT04}, \cite{KTT06a}, and \cite{KTT06b}). Inspired by Kirillov's proof that stretched Kostka numbers are polynomial in the stretching factor $N$ for fixed partitions, King, Tollu, and Toumazet made a similar conjecture for stretched Littlewood-Richardson numbers.

\begin{conjecture}[\cite{KTT04}]
For all partitions $\lambda, \mu,$ and $\nu$, there exists a polynomial $P^\nu_{\lambda, \mu}(N)$ in $N$ with nonnegative rational coefficients such that $P^\nu_{\lambda, \mu}(1) =c^\nu_{\lambda,\mu}$ and $P^\nu_{\lambda, \mu}(N) = c^{N\nu}_{N\lambda, N\mu}$ for all positive integers $N$. 
\end{conjecture}

Along with this, Fulton conjectured that if $P(1) = 1$, then $P(N) = 1$ for all $N \geq 1$, while King, Tollu, and Toumazet conjectured that if $P(1) =2$, then $P(N) = N+1$. Derksen and Weyman proved the polynomiality conjecture \cite{DW02} and Rassart \cite{Ras04} proved it again shortly afterwards, while Fulton's conjecture was first proven combinatorially by Knutson, Tao, and Woodward \cite{KTW04}, then later geometrically by Belkale \cite{Bel07}, and again using quivers \cite{DW11}. The conjecture that $P(N) =N+1$ when $P(1)=2$ was first proven combinatorially by Ikenmeyer \cite{Ike16} and then geometrically and by using quivers by Sherman in \cite{She15} and \cite{She17}, respectively.

 In this section we explicitly compute the stretched function for certain weights for the sun quiver and verify that similar statements hold for the respective generalized Littlewood-Richardson coefficients. We also note that while similarly defined weights for the star quiver all lie on extremal rays of the cone of effective weights, this is not true for our case.

\subsection{Level-$1$ weights}
In \cite{Fei15}, Fei defines a weight for the star quiver 
to be level-$m$ if the weight has value $m$ at the central vertex. In Lemma $2.3$ of the paper, he classifies all level-$1$ effective weights and shows that they lie on an extremal ray.  For the star quiver, the Littlewood-Richardson coefficient arising from any effective level-1 weight is  of the form $c^{1^{i+j}}_{1^i, 1^j}$. We use this idea to describe similar weights for the sun quiver.

We define a \bi{level-1 weight} for the sun quiver to be a (for now, not necessarily effective) nonzero weight with at most one nonzero entry along any flag, with the nonzero entry being $1$ for the flags going out and $-1$ for the flags going in. This will correspond to at most one jump along each flag for the defining dimension vector.  Because $\sigma(\beta)=0$ is a necessary condition, if $j_1, \ldots, j_m$ are the vertices along the flags for which $\sigma(j_i) \neq 0$, counting the vertices towards the central ones (so $j_i$ is vertex $(j_i,i)$), with $j_i=0$ to mean that the weight is trivial along flag $i$, then $\sum_{i \text{ odd}} j_i = \sum_{i \text{ even}} j_i$ is necessary, though not sufficient.

 We'll find it useful in this section to describe the effective weights directly in terms of the jumping numbers $j_i$ rather than by the conditions found in Proposition \ref{eigenvalue-sun}.
 
\begin{lemma}\label{level-1-weight}
Let $\sigma$ be a level-$1$ weight for the sun quiver $Q$ and let $j_1, \ldots, j_m$ be the vertices along flags $1, \ldots, m$ for which $\sigma(j_i) \neq 0$ with $j_i = 0$ if $\sigma$ is trivial on flag $i$. Then the following are equivalent:
\begin{enumerate}
\item $\sigma \in \cone$;\vspace{.7pt}
\item   $\dsp \sum_{i \text{ odd}} j_i = \sum_{i \text{ even}} j_i$ and $j_i - j_{i+1} + j_{i+2} \geq 0$ for $1 \leq i \leq m$, where $j_{m+1} = j_1$ and $j_{m+2} = j_2$.
\end{enumerate}
\end{lemma}

\begin{proof}
 Because the partition arising from a flag with the only nonzero weight being $1$ or $-1$ at vertex $j_i$ is $(1^{j_i})$ if $j_i \neq 0$ and $(0)$ if $j_i = 0$, the generalized Littlewood-Richardson coefficient arising from this weight is
\[
\dsp f((1^{j_1}), \ldots, (1^{j_m}))= \sum c^{(1^{j_1})}_{\alpha_1,\alpha_2}\cdot c^{(1^{j_2})}_{\alpha_2,\alpha_3}\cdots  c^{(1^{j_{m-1}})}_{\alpha_{m-1}, \alpha_m} \cdot c^{(1^{j_m})}_{\alpha_m, \alpha_1}.
\]
As observed above, $\sum_{i \text{ odd}} j_i = \sum_{i \text{ even}} j_i$ is equivalent to $\sigma(\beta) = 0$ and because 
\[
\alpha_i = j_i - \alpha_{i+1} =j_i - j_{i+1} + \alpha_{i+2} =  j_i - j_{i+1} + j_{i+2} - \alpha_{i+3},
\]
the condition $j_i - j_{i+1} + j_{i+2} \geq 0$ is necessary. It's easy to check that the inequalities on the $j_i$'s imply that $j_i \leq j_{i+1}$ and $j_{i-1} \leq j_{i-2}$ for some $i$, which will be sufficient to prove that such a weight is effective. After reindexing, suppose $j_1 \leq j_2$ and $j_m \leq j_{m-1}$. Then the choice $\alpha(1)=0$ uniquely determines the other $\alpha(i)$ and it follows from the above conditions that each $\alpha(i) \geq 0$, making  the weight effective.
\end{proof}

\begin{remark}
There are fewer conditions on the $j_i$ for the level-$1$ weights to be effective than those stated for general weights in the context of generalized eigenvalue problem for Hermitian matrices in Proposition  \ref{eigenvalue-sun}. This is because in the general case we can't say that either $H(i) - H(i+1)$ or $H(i+1)-H(i)$ is positive semi-definite, while we can make such a direct comparison of $j_i$ and $j_{i+1}$. This allowed us to have an $i$ such that $j_i \leq j_{i+1}$ and $j_{i-1} \leq j_{i-2}$, proving that such a weight was effective. However, this at least tells us that if there is an $i$ such that $H(i+1) - H(i)$ and $H(i-2) - H(i-1)$ are positive semi-definite (along with the other conditions on the $H(i)$), then the $H(i)$ solve the generalized eigenvalue problem. These conditions, though, are not necessary.
\end{remark}

\begin{remark}
As opposed to the case for the star quiver, not every effective level-1 weight lies on an extremal ray for the sun quiver. We found several such weights lying on the facets in the case $n=2, m=6$. It can be checked that the first weight in the \hyperref[appendix]{Appendix}
provides an instance of a level-1 weight on an extremal ray while the second weight in the first row 
provides one which is not. 
\end{remark}

We now want to determine the value $\dm \si{\beta}_\sigma$ for a level-1 weight $\sigma$.

\begin{lemma}\label{value-level-1}
Let $\sigma \in \cone$ be a level-1 weight for the sun quiver $Q$. Let $j_1,\ldots, j_m$ be the jumping numbers defining the weight and define $J_i = j_i-j_{i+1}+j_{i+2}$, where $j_{m+1} = j_1$ and $j_{m+2} = j_2$. If $s=\min\{j_i,J_i \mid 1 \leq i \leq m\}$, then $\dm \si{\beta}_\sigma = s+1$. 
\end{lemma}

\begin{proof}
We first show that $\dm \si{\beta}_\sigma \leq s+1$. Throughout we will denote $(1^{j_i})$ as simply $j_i$. Clearly, any choice of some $\alpha_i$ completely determines each of the other $\alpha_j's$. Moreover, because each partition is of the form $(1^{j_k})$, $c^{j_k}_{\alpha_k, \alpha_{k+1}} = 1$ whenever it's nonzero.  If $i$ is such that $J_i$ or $j_i$ is minimal among the set, then consider the factors
\[
c^{j_i}_{\alpha_i,\alpha_{i+1}} \cdot c^{j_{i+1}}_{\alpha_{i+1},\alpha_{i+2}} \cdot c^{j_{i+2}}_{\alpha_{i+2},\alpha_{i+3}}
\]
in the summation. Then $\alpha_i \leq j_i$ and similarly because $\alpha_{i+3} = j_{i+2}-j_{i+1}+j_i - \alpha_i = J_i - \alpha_i$, we must have $\alpha_i \leq J_i$ in order for this factor to be nonzero. Hence, there are at most $s+1$ choices for $\alpha_i$ resulting in this factor being nonzero.

Suppose $s =j_i$ for some $i$. To show equality, we only need to show that $\alpha_k \geq 0$ for each $k$ for each choice of $\alpha_i \in \{0,\ldots, s\}$. This is quickly done since $i$ was chosen to be such that $j_i \leq j_k$ and $j_i \leq J_k$ for each $k$ along with $\alpha_i \leq j_i$. Similarly, if $s = J_i$, each $\alpha_k$ will be nonnegative after noticing that $j_i \leq j_{i+1}$ in this case since $J_i \leq j_{i+2}$. Thus, the only choices for $\alpha_i$ resulting in a nonzero term in the summation are $0,1,\ldots, s$, and each such choice results in adding one to the summation.
\end{proof}

\subsection{Stretched weights}

For a level-1 weight $\sigma$, we're interested in the stretched weights $N \sigma$, $N \in \Z^+$. If $j_1, \ldots, j_m$ are the jumping numbers of $\sigma$, the corresponding partitions will be $(N^{j_i})$. Because $|(N^{j_i})| = Nj_i$ for all $i$, Lemma \ref{level-1-weight} generalizes immediately to stretched level-1 weights. Similarly, Lemma \ref{value-level-1} generalizes in this case because of the following lemma which is quickly checked by using the Littlewood-Richardson rule.

\begin{lemma}\label{rectangular-partitions-1}
Let $\lambda, \mu$ be partitions and $\nu = (N^n)$ a rectangular partition. Then $c^\nu_{\lambda, \mu}$ is either 0 or 1. It is equal to 1 if and only if $\lambda_i + \mu_{n+1-i}=N$ for $i=1,\ldots, n$.
\end{lemma}

In the proof of the next lemma, we use a partial ordering on the set of rectangular partitions $(N^n)$ for a fixed $N$ defined by  $\lambda_1 \leq \lambda_2$  to mean that the Young diagram of $\lambda_1$ fits inside that of $\lambda_2$, meaning $\lambda_2 - \lambda_1$ is a partition. With this, $\lambda_1 \dotplus \lambda_2$ means stacking the corresponding Young diagrams on top of each other (or in terms of partitions, $(N^{n_1}) + (N^{n_2}) = (N^{n_1+n_2}))$, so $\lambda_3 \leq \lambda_1 \dotplus \lambda_2$ means the Young diagram of $\lambda_3$ fits inside the stacked diagrams of $\lambda_1$ and $\lambda_2$, or equivalently, $(N^{n_1 + n_2 - n_3})$ is a partition.  We will use the notation $\lambda_1 \dotplus (-\lambda_2)$ to mean we are instead subtracting $\lambda_2$ from the bottom of the diagram of $\lambda_1$.

\begin{lemma}
\label{effective stretched weight}
For a level-1 weight $\sigma$ for the sun quiver, let $N \in \Z^+$, $j_1, \ldots, j_m$ the corresponding jumping numbers, and $J_i = j_i - j_{i+1} + j_{i+2}$ for $i=1, \ldots, m$, with $j_{m+1} = j_1$ and $j_{m+2} = j_2$. The following are equivalent:
\begin{enumerate}
\item  $N\sigma \in C(Q,\beta)$;\vspace{.7pt}
\item $\dsp \sum_{i \text{ odd}} j_i  = \sum_{i \text{ even}} j_i$ and $J_i \geq 0$ for all $i$.
\end{enumerate}
 If $N\sigma$ is effective, then $\dm \si{\beta}_{N\sigma} = {N+s \choose N}$, where  $s= \min\{j_i, J_i \mid 1 \leq i \leq m\}$.

\end{lemma}

\begin{proof}
As mentioned above, the necessary and sufficient conditions for $N \sigma$ to be effective are proven the same way as in Lemma \ref{level-1-weight}. We adapt the proof of Lemma \ref{value-level-1} to compute the value of the dimension of the weight space for the stretched case. 

First suppose $s=j_i$, and without loss of generality, suppose $i=1$. The number of partitions $\alpha_1$ such that $\alpha_1 \leq (N^{j_1})$ is ${N+j_1 \choose N}$. This is because if we consider the Young diagram corresponding to $(N^{j_1})$, then we choose how many entries of $\alpha_1$ have value $N$, then how many have value $N-1$, and so on, which is the same as choosing where to place $N$ dividers among $j_1$ entries. Because each $\alpha_1$ uniquely determines the other $\alpha_i$ and because each $c^{(N^{j_i})}_{\alpha_i, \alpha_{i+1}}$ is equal to one when nonzero by  Lemma \ref{rectangular-partitions-1}, dim $\si{\beta}_{N\sigma} \leq {N + j_1 \choose N}$. The other direction is proved in a similar way as Lemma \ref{value-level-1}.
\end{proof}

\begin{remark}
The value of $f(\lambda(1),\ldots,\lambda(m))$ is independent of the value of $n$, the length of each flag. This number can only enlarge the value of the coefficient, which is instead determined by the smallest $j_i$ or $J_i$. This formula also agrees with the value that we found in Lemma \ref{value-level-1} since in that case $N=1$.
\end{remark}

For a fixed weight $\sigma$ for the sun quiver, we showed in Lemma \ref{si-saturation} that for each $N \geq 1$, 
\[
f(N\lambda(1),\ldots, N\lambda(m)) = \dim \si{\beta}_{N\sigma},
\]
where $\lambda(1),\ldots, \lambda(m)$ are the partitions arising from $\sigma$ as stated in equation (\ref{weight-sigma_1}). Clearly, this is a polynomial as each stretched function of a single Littlewood-Richardson coefficient is a polynomial. The above formula allows us to calculate $\dm \si{\beta}_{N \sigma}$ for any level-1 weight immediately.

\begin{proposition}
\label{LR-sun-polynomial}
Let $\lambda(1),\ldots, \lambda(2k)$, $k \geq 2$, be partitions of at most $n$ parts and of the form $(1^{j_i})$ if $j_i \neq 0$ and zero if $j_i = 0$ for some integers $0\leq j_1, \ldots, j_m \leq n$. Suppose the $j_i$ satisfy the following  conditions:
\begin{enumerate}
\item $\dsp \sum_{i \text{ odd}} j_i = \sum_{i \text{ even}} j_i$;\\
\item $\dsp J_i := j_i - j_{i+1} + j_{i+2} \geq 0$, where $j_{2k+1} = j_1, \; j_{2k+2} = j_2$.
\end{enumerate}
Then for any $N \in \Z^+$, the stretched Littlewood-Richardson polynomial $f(N\lambda(1),\ldots, N\lambda(2k))$ is equal to ${N + s \choose N}$, where $s = \min \{j_i,J_i \mid 1 \leq i \leq 2k\}$. If either (1) or (2) is not satisfied, then $f(N\lambda(1),\ldots, N\lambda(2k)) = 0$.
\end{proposition}

\begin{proof}
This follows immediately from Lemma \ref{effective stretched weight} and Theorem  \ref{saturation-sun}.
\end{proof}

\begin{remark}
Proposition \ref{LR-sun-polynomial} shows that the conjectures of King, Tollu, and Toumazet, and of Fulton are true for these generalized coefficients in the cases that the partitions are of the stated forms. Namely, as a function of $N \in \Z^+$,  the stretched Littlewood-Richardson function is a polynomial $P$ and whenever $P(1) = 1$, $P(N) =1$, and when $P(1) = 2$, $P(N) = N+1$. In addition, we showed the saturation property (Theorem \ref{saturation-sun}) saying $P(1) = 0$ implies $P(N) = 0$. Furthermore, it has been conjectured (\cite{KTT04}, Conjecture 3.3) that $P(1)=3$ implies $P(N)$ is either $2N+1$ or ${N+2 \choose N}$, which also agrees with our results. It would be interesting to see if similar conjectures for these generalized coefficients hold for all weights.
\end{remark}

\vspace{.2in}

\section{Polytopal description and complexity}
\label{polytope}
In this section we examine the complexity of the branching multiplicity by defining a polytope  whose number of lattice points is equal to the multiplicity. The main result is Theorem \ref{sun-complexity}, which states that the positivity of the multiplicity, that is, whether or not it is zero, can be calculated in strongly polynomial time.

\subsection{Geometric complexity theory}

Geometric complexity theory (GCT) was introduced by Mulmuley and Sohoni in a series of papers (see \cite{MS07}, \cite{MS01a}, \cite{MS01b}, \cite{MS08}, \cite{MNS12},  \cite{MS17}, \cite{Mul10}, \cite{Mul11}) in the early 2000's with the purpose of approaching fundamental problems in complexity theory, such as P vs. NP, through algebraic geometry and representation theory. Previously, \cite{KT01} and \cite{LM06} had independently shown that the positivity of Littlewood-Richardson coefficients could be computed in polynomial time while \cite{Nar05} had shown that the actual computation of these numbers was a \#P-complete problem, the complexity class for problems for which (unless P=NP) there does not exist a polynomial time algorithm for computing them (rather, it takes an exponential time in the worst case), and such that the computation is at least as difficult as every P problem. 

The following is the main theorem of \cite{MS05}, where the \emph{bit length} of a partition $\lambda = (\lambda_1, \ldots, \lambda_k), \, \lambda_k > 0, $ is the bit length of the specifications: $\sum_{i=1}^k\log_2 \lambda_i$. 

\begin{theorem}
Deciding whether $c^{\nu}_{\lambda,\mu}$ is positive can be computed in strongly polynomial time in the sense of \cite{GLS93}. This means that the number of arithmetic steps is polynomial in the number of positive parts of $\nu$ (say $n$), does not depend on the bit lengths of $\lambda_i, \mu_j, \nu_k$, and the bit length of every intermediate operand that arises in the algorithm is polynomial in the total bit length of $\lambda, \mu, \nu$. 
\end{theorem}

In fact, by attaching zeros to the partitions, one can subsume the dependence on $n$ into the bit lengths of $\lambda,\mu,$ and $\nu$. This is especially amazing as the the dimension of the Weyl module $S^\nu(V)$ is exponential in $n$ and the bit lengths of the $\nu_k$'s, yet deciding if an exponential dimensional object $S^\nu(V)$ arises in the decomposition of another exponential dimensional object $S^\lambda(V) \otimes S^\mu(V)$ can be decided in time that is polynomial in only $n$ and the bit lengths of the labels $\lambda,\mu,$ and $\nu$.

Because of results such as these along with the ubiquity of the plethysm problem and related problems in representation theory, GCT allows one to compare the complexity of several problems. The proof of deciding the positivity of a Littlewood-Richardson coefficient relies on two main points: a polyhedral interpretation of these numbers and the saturation theorem. While we define a polytope for the generalized Littlewood-Richardson coefficients to prove a similar result, it would be nice to have a purely combinatorial algorithm, such as those of max-flow or weighted matching problems in combinatorial optimization. Much work has been made towards this for single Littlewood-Richardson coefficients (see \cite{BI09}, \cite{BI13}, and \cite{Ike16}).

\subsection{Polytopal description}

In order to determine the complexity of the positivity of  multiplicity (\ref{one}), we will define a polytope by determining a system of homogeneous linear inequalities whose number of integer-valued solutions is precisely the  multiplicity. The idea is to use the Littlewood-Richardson hives defined by Knutson and Tao in \cite{KT99}. 

To define the polytope associated with multiplicity (\ref{one}),  subdivide a regular $m$-gon into $m$ triangles with $n+1$ vertices along each exterior edge and a common vertex at the center. Subdivide each of these triangles into $n^2$ triangles of the same size, so the hexagon is divided into $mn^2$ total triangles. We label the edges in the $r^{th}$ triangular array in the following way: the first subscript $i$ refers to the row from bottom to top while the second subscript $j$ refers to the diagonal from left to right, and $0 \leq i,j \leq n-1$. The edges along increasing diagonals are labeled $e_{ij},$ the edges along decreasing diagonals are labeled $f_{ij}$, and the horizontal edges  are $g_{ij}$. The superscript $r$ refers to which triangular array is used, though this is often neglected. For instance, when $n=3$ the $r^{th}$ triangular array is labeled 
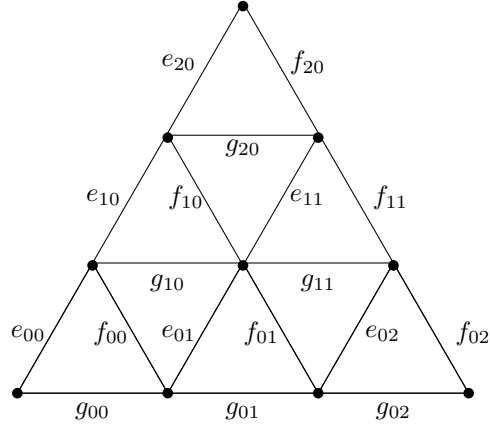
\[
\begin{tikzpicture}
    \draw  (0,0) -- ++(0:6) -- ++(120:6) -- cycle;
    \draw (0,0) -- ++(0:2) -- ++(120:2) -- cycle;
    \draw (2,0) -- ++(0:2) -- ++(120:2) -- cycle;
    \draw (4,0) -- ++(0:2) -- ++(120:2) -- cycle;
    \draw (2,0) -- ++(60:2) -- ++(180:2) -- cycle;
    \draw (4,0) -- ++(60:2) -- ++(180:2) -- cycle;
    \draw (3,1.7) -- ++(60:2) -- ++(180:2) -- cycle;
    
\coordinate [label={below left: }] ($a_{00}$) at (0,0);
\coordinate [label={below: }]  ($a_{10}$) at (2,0);
\coordinate [label={below:}]  ($a_{20}$) at (4,0);
\coordinate [label={below:}] ($a_{30}$) at (6,0);
\coordinate [label={left:}]  ($a_{01}$) at (1,1.7);
\coordinate [label={below: }]  ($a_{11}$) at (3,1.7);
\coordinate [label={right:}]  ($a_{21}$) at (5,1.7);
\coordinate [label={left:}]  ($a_{02}$) at (2,3.4);
\coordinate [label={right:}]  ($a_{12}$) at (4,3.4);
\coordinate [label={above:}]  ($a_{03}$) at (3,5.15);

 \foreach \i in {$a_{00}$,$a_{01}$,$a_{02}$,$a_{03}$,$a_{10}$,$a_{11}$,$a_{12}$,$a_{20}$,$a_{21}$,$a_{30}$}
 \fill (\i) circle (2pt);
 
 \coordinate [label={below: $g_{00}$ }] () at (1,0);
\coordinate [label={below: $g_{01}$}]  () at (3,0);
\coordinate [label={below: $g_{02}$}]  () at (5,0);
\coordinate [label={below: $g_{10}$}] () at (2,1.7);
\coordinate [label={below: $g_{11}$}]  () at (4,1.7);
\coordinate [label={below: $g_{20}$ }]  () at (3,3.5);
\coordinate [label={left: $e_{00}$}]  () at (.5,.8);
\coordinate [label={left: $f_{00}$}]  () at (1.6,.8);
\coordinate [label={left: $e_{01}$}]  () at (2.5,.8);
\coordinate [label={left: $f_{01}$}]  () at (3.6,.8);
\coordinate [label={right: $e_{02}$}] () at (4.5,.8);
\coordinate [label={right: $f_{02}$}] () at (5.7,.8);
\coordinate [label={left: $e_{10}$}]  () at (1.5,2.6);
\coordinate [label={left: $f_{10}$}]  () at (2.6,2.6);
\coordinate [label={right: $e_{11}$}] () at (3.5,2.6);
\coordinate [label={right: $f_{11}$}] () at (4.6,2.6);
\coordinate [label={left: $e_{20}$}] () at (2.5,4.4);
\coordinate [label={right: $f_{20}$}] () at (3.5,4.4);
\end{tikzpicture}
\]

Let $E$ be the set of hive edges and $\RR^E$ the labelings of these edges by real numbers. There are three ways that two adjacent triangles inside a single triangular array can form a rhombus: 
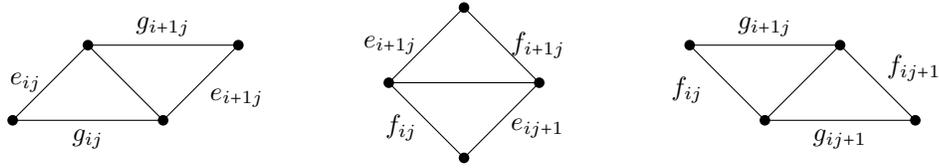
\[
\begin{tikzpicture}

\coordinate [label={above: $g_{i+1j}$}] () at (-2,.5);
\coordinate [label={below: $g_{ij}$}] () at (-3,-.5);
\coordinate [label={left: $e_{ij}$}] () at (-3.5,0);
\coordinate [label={right: $e_{i+1j}$}] () at (-1.5,-.2);
\coordinate [label={below left:}] (a) at (-4,-.5);
\coordinate [label={above left:}] (b) at (-3,.5);
\coordinate [label={below right:}] (c) at (-2,-.5);
\coordinate [label={above:}] (d) at (-1,.5);
\draw (a)--(b)--(d)--(c)--cycle;
\draw (b)--(c);

\foreach \i in {a,b,c,d}
  \fill (\i) circle (2pt);
  
\coordinate [label={left: $e_{i+1j}$}] () at (1.5,.5);
\coordinate [label={right: $f_{i+1j}$}] () at (2.5,.5);
\coordinate [label={left: $f_{ij}$}] () at (1.5,-.6);
\coordinate [label={right: $e_{ij+1}$}] () at (2.5,-.6);
\coordinate [label={above:}] (a) at (2,1);
\coordinate [label={left:}] (b) at (1,0);
\coordinate [label={right:}] (c) at (3,0);
\coordinate [label={below:}] (d) at (2,-1);
\draw (a)--(b)--(d)--(c)--cycle;
\draw (b)--(c);

\foreach \i in {a,b,c,d}
  \fill (\i) circle (2pt);
 
\coordinate [label={left: $f_{ij}$}] () at (5.3,-.1);
\coordinate [label={above: $g_{i+1j}$}] () at (6,.5);
\coordinate [label={below: $g_{ij+1}$}] () at (7,-.5);
\coordinate [label={right: $f_{ij+1}$}] () at (7.5,0.2); 
\coordinate [label={above:}] (a) at (5,.5);
\coordinate [label={above right:}] (b) at (7,.5);
\coordinate [label={below: }] (c) at (6,-.5);
\coordinate [label={below right:}] (d) at (8,-.5);
\draw (a)--(b)--(d)--(c)--cycle;
\draw (b)--(c);

\foreach \i in {a,b,c,d}
  \fill (\i) circle (2pt);
\end{tikzpicture}
\]

We say these rhombi satisfy the \bi{rhombus inequalities} if for each triangle and rhombus appearing, we have

\begin{equation}
\label{rhombus-edge-inequalities}
e_{ij} \geq e_{ij+1}, \;\;\; g_{ij} \geq g_{i+1 j}; \qquad 
f_{i+1 j} \geq f_{ij}, \;\;\;  e_{i j+1} \geq e_{i+1 j}; \qquad
f_{ij} \geq f_{i j+1}, \;\;\; g_{i+1 j} \geq g_{i j+1};
\end{equation}
\[
e_{ij} + f_{ij} = g_{ij}, \qquad \quad e_{i \, j+1} + f_{ij} = g_{i+1 \,j}.
\]

Define an \bi{$(m,n)$-LR sun hive} to be a regular $m$-gon subdivided into $m$ triangular arrays with $n+1$ vertices along each edge that satisfies the rhombus inequalities and the border conditions
\begin{equation}
\label{boundary-edge-conditions}
\sum_{i=0}^{n-1} e^r_{i0} + \sum_{i=0}^{n-1} f^r_{i n-i} = \sum_{j=0}^{n-1} g_{0j} 
\end{equation}
for each $1 \leq r \leq m$. It is \bi{integral} if the labeling lies in $\Z^E$. These inequalities define a convex polyhedral cone, denoted $C \subseteq \RR^E$. An \bi{LR hive} is a single triangular array that satisfies the rhombus inequalities and border conditions for that array, so an $(m,n)$-LR sun hive consists of $m$ LR hives with $n$ edges along each side of the boundary of a regular $m$-gon and with the respective conditions and appropriately shared sides.  Let $B$ be the set of border edges $g^k_{0j}$ for $1\leq r \leq m, \; 0 \leq j \leq n-1$, and $\rho: \RR^E \to \RR^B$ the restriction map of an LR sun hive to its border. For each $b \in \RR^B$, the fiber $\rho^{-1}(b) \cap C$ is a compact polytope, called the \bi{$m$-sun hive polytope} over $b$. 

We recall the main result of \cite{KT99}. For three $n$-tuples
\[
\lambda = (\lambda_1,\ldots, \lambda_n), \qquad 
\mu = (\mu_1,\ldots, \mu_n), \qquad
\nu = (\nu_1, \ldots, \nu_n)
\]
that satisfy the boundary condition $|\nu| = |\lambda| + |\mu|$, the triangular array with border determined by $\lambda, \mu, \nu$ is the one with specified border edges
\[
\begin{tikzpicture}

 \coordinate  (l0) at (10,0);
\coordinate   (n1) at (12,0);
\coordinate  (n2) at (14,0);
\coordinate  (n3) at (16,0);
\coordinate  (n4) at (18,0);
\coordinate  (l1) at (11,1.7);
\coordinate (ln1) at (13,1.7);
\coordinate (ln2) at (15,1.7);
\coordinate  (l2) at (12,3.4);
\coordinate (ln3) at (14,3.4);
\coordinate  (l3) at (13,5.1);
\coordinate  (l4) at (14,6.8);
\coordinate  (m1) at (15,5.1);
\coordinate  (m2) at (16,3.4);
\coordinate  (m3) at (17,1.7);
\draw (l0)--(l4)--(n4)--cycle;
\draw (l3)--(m1)--(ln3)--(m2)--(ln2)--(m3)--(n3)--cycle;
\draw (l3)--(ln3)--(l2)--(ln1)--(n1)--(l1)--cycle;
\draw (ln3)--(ln2)--(ln1)--cycle;
\draw (ln1)--(ln2)--(n2)--cycle;
\draw (l1)--(ln1)--(n1)--cycle;

 \foreach \i in {l0,l1,l2,l3,l4,n1,n2,n3,n4,m1,m2,m3,ln1,ln2,ln3}
 \fill (\i) circle (2pt);
 
 \tkzLabelSegment[left=2pt](l0,l1){$\lambda_1$}
 \tkzLabelSegment[left=2pt](l1,l2){$\lambda_2$}
 \tkzLabelSegment[left=2pt](l2,l3){$\ldots$}
 \tkzLabelSegment[left=2pt](l3,l4){$\lambda_n$}
 \tkzLabelSegment[right=2pt](l4,m1){$\mu_1$}
 \tkzLabelSegment[right=2pt](m1,m2){$\mu_2$}
 \tkzLabelSegment[right=2pt](m2,m3){$\ldots$}
 \tkzLabelSegment[right=2pt](m3,n4){$\mu_n$}
 \tkzLabelSegment[below=2pt](l0,n1){$\nu_1$}
 \tkzLabelSegment[below=2pt](n1,n2){$\nu_2$}
 \tkzLabelSegment[below=2pt](n2,n3){$\ldots$}
 \tkzLabelSegment[below=2pt](n3,n4){$\nu_n$}

\end{tikzpicture}
\]

\begin{theorem}[\cite{KT99}, Theorem 4]
	\label{LR-hives}
	The Littlewood-Richardson coefficient $c^\nu_{\lambda, \, \mu}$ is the number of integer LR hives with boundary labels determined by $\lambda, \mu$, and $\nu$.
\end{theorem}

Because we are only interested in integer LR hives it suffices to restrict to when $\lambda,\mu,\nu$ are partitions. Further, if each partition has at most $n$ parts, then the LR hives are LR $n$-hives. Though the border conditions are obvious from the necessary condition that $|\nu| = |\lambda| + |\mu|$ for $c^\nu_{\lambda,\mu}$ to be nonzero, the rhombus inequalities seem mysterious at first. Their inspiration comes from Gelfand-Tsetlin patterns and Cauchy's Interlacing Theorem for the eigenvalues of Hermitian matrices. The first two pairs of inequalities ensure that the tuples are weakly decreasing while the third pair gives a way of associating a contratableau satisfying the Littlewood-Richardson rule; for more on this, see \cite{Buc00}.

Because one LR hive is used to calculate one Littlewood-Richardson coefficient, it stands to reason that ``gluing" multiple LR hives together appropriately should be used to calculate our generalized Littlewood-Richardson coefficients. Before stating and proving this we first make precise how we intend to ``glue" the LR hives. Given one LR hive we combine it with another LR hive by requiring the two to share a side other than the base. This results in the second LR hive being flipped. Of course, we need to verify the values assigned to the shared edges coincide and make precise the edge labelings along with the rhombus inequalities for the flipped hive. For the flipped hive the edges on a descending diagonal are now labeled by the $e$'s while the ascending diagonal edges are labeled by the $f$'s. 

Under this notation, $e^{k+1}_{j0} = f^{k}_{n-1-j \, j}$. Flipping the triangular arrays causes each type of rhombus to be flipped, but by also switching the labels for the $e$'s and $f'$s the same rhombus inequalities in (\ref{rhombus-edge-inequalities}) hold. For these flipped arrays we will want to specify when $n$-tuples $\lambda, \mu, \nu$ such that $|\nu| = |\lambda| + |\mu|$ determine the border and align along shared edges, but we will wait to do this depending on which side of the flipped hive we want to have labeled $\nu$.

 \begin{remark}
 When defining the $(m,n)$-LR sun hive the rhombus inequalities did not include those two types arising from rhombi of adjacent triangles from different hives. Because one hive is flipped, there is no natural way of determining which direction the inequality should be and the direction may differ in different examples. Interestingly, though, the direction of the inequalities arising from adjacent hives is the same within each individual example examined.
\end{remark}

With this notation we may now prove the polytopal description of the generalized Littlewood-Richardson coefficient (\ref{one}).

\begin{theorem}
\label{sun hive}
For partitions $\lambda(1), \ldots, \lambda(2k)$, $k \geq 2$, of no more than $n$ parts, the generalized Littlewood-Richardson coefficient 
\[
\sum c^{\lambda(1)}_{\alpha(1),\alpha(2)} c^{\lambda(2)}_{\alpha(2),\alpha(3)} \cdots c^{\lambda(2k-1)}_{\alpha(2k-1),\alpha(2k)} c^{\lambda(2k)}_{\alpha(2k),\alpha(1)}
\]
is equal to the number of integer $(2k,n)$-LR sun hives with external boundary labels determined by the $\lambda(i)$ in cyclic orientation (so that the edge labeled $\lambda(r)$ is between the edges labeled $\lambda(r+1)$ and $\lambda(r-1)$). For instance,  the boundary labels of a $(6,n)-$LR sun hive are
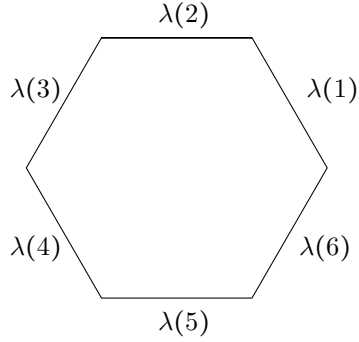
\[
\begin{tikzpicture}
   \draw (-2,2) -- ++(0:2) -- ++(300:2) -- ++(240:2)  -- ++(180:2) -- ++(120:2) --++(60:2) --++(0:2) -- cycle;
	
\coordinate [label={above:$\lambda(2)$}] () at (-.9,2);
\coordinate [label={right:$\lambda(1)$}] () at (.6,1.3);
\coordinate [label={right: $\lambda(6)$}] () at (.5,-.8);
\coordinate [label={below:$\lambda(5)$}]  () at (-.9,-1.5);
\coordinate [label={left: $\lambda(4)$}] () at (-2.4,-.8);
\coordinate [label={left:$\lambda(3)$}]  () at (-2.4,1.3);
	
\end{tikzpicture}
\]
\end{theorem}

\begin{proof} 
By Theorem \ref{LR-hives} the Littlewood-Richardson coefficient $c^{\lambda(r)}_{\alpha(r),\alpha(r+1)}$ is equal to the number of integer LR $n$-hives with $\lambda(r)$ as the base which satisfy the boundary conditions and rhombus inequalities, where $\alpha(r), \alpha(r+1)$ are some tuples of no more than $n$ parts forming the other two sides of the $r^{th}$ triangular array. Necessarily the tuple $\alpha(r)$ is also a boundary of the $(r-1)^{th}$ triangular array while $\alpha(r+1)$ is a boundary of the $(r+1)^{th}$. We use the previously defined notation for each hive and adjacent (flipped) hive, so we only need to specify the border labels. If the base labeled $\lambda(r)$ has edges labeled $\lambda(r)_1, \ldots, \lambda(r)_n$ from \textit{left to right}, then the adjacent base labeled $\lambda(r+1)$ has edges labeled $\lambda(r+1)_1, \ldots, \lambda(r+1)_n$ from \textit{right to left}. In this way, edges labeled by $\alpha(r)$ and $\alpha(r+1)$ in the $r^{th}$ LR hive are $\alpha(r)_1, \ldots, \alpha(r)_n, \alpha(r+1)_1, \ldots, \alpha(r+1)_n$ clockwise while the edges in the adjacent hive are labeled $\alpha(r+2)_1,\ldots, \alpha(r+2)_n, \alpha(r+1)_1,\ldots, \alpha(r+1)_n$ counterclockwise. The multiplicity
\[
\sum c^{\lambda(1)}_{\alpha(1),\alpha(2)} c^{\lambda(2)}_{\alpha(2),\alpha(3)} \cdots c^{\lambda(2k-1)}_{\alpha(2k-1),\alpha(2k)} c^{\lambda(2k)}_{\alpha(2k),\alpha(1)}
\]
is then equal to the number of integer $(2k,n)$-LR sun hives with these choices of $\alpha(1),\ldots, \alpha(2k)$. The total number of integer $(2k,n)$-LR sun hives with the boundaries $\lambda(1),\ldots, \lambda(2k)$ is then the sum over all possible integer tuples $\alpha(1), \ldots, \alpha(2k)$ with at most $n$ parts. 

\end{proof}

\begin{remark}
There is a characterization of LR hives with vertex labels rather than edge labels. Though the two labelings may be used interchangeably for all results concerning a single Littlewood-Richardson coefficient, the vertex labeling fails in the case of the generalized coefficients because the vertices along a shared boundary would not necessarily agree. For instance, the vertex at the center of the regular $n$-gon could only be zero while this would force the external boundary labels to not be the $\lambda(i)$.
\end{remark}

 \begin{remark}
 In the previous theorem, it is necessary that the number of partitions be at least four and even. We saw that adjacent LR hives must ``flip" in order to line up the boundary edges and in our description the edges on the side determined by $\lambda(i)$ are labeled by $\lambda(i)_1, \ldots, \lambda(i)_n$ from left to right for odd $i$ and in reverse order for even $i$. If the number of partitions, $m$, were odd, then the first and $m+1$ hives are the same, yet these have different parities, so we get two different labelings. The number of hives must then be even and it is easily checked that $m=2$ fails.
 \end{remark}

 As we've seen, the rhombus inequalities (\ref{rhombus-edge-inequalities}) and boundary conditions (\ref{boundary-edge-conditions}) determine a polytope whose number of lattice points corresponds to the multiplicity (\ref{one}) when the external boundaries are determined by the $\lambda(r)$. For each $1 \leq r \leq 2k$, these inequalities may be solved into a linear program $A_r \mathbf{x_r} \leq \mathbf{b_r}$, where $A_r$ is a matrix with entries $0,1,-1$, $\mathbf{x_r}$ is the vector of interior edges $e^r_{ij},\, f^r_{ij}, \, g^r_{ij} \; 0 \leq i \leq n-1, \, 0 \leq j \leq n-1$,
  and  the entries of $\mathbf{b_r}$ are homogeneous, linear forms in the entries of $\lambda(r)$ (which are integral when $\lambda(r)$ is a partition). Because this can be done for each $r$, we can express all of these as a single linear program $A \mathbf{x} \leq \mathbf{b}$, where $A$ is the block sum of the matrices $A_r$ and similarly for $\mathbf{x}$ and $\mathbf{b}$. Again, $\mathbf{b}$ will be homogeneous, which is necessary for the proof of the complexity of the positivity of the generalized Littlewood-Richardson coefficient. In this way, the multiplicity is equal to the number of integer-valued vector solutions $\mathbf{x}$ to this inequality. This proves the following.

\begin{lemma}
\label{linear-program}
For partitions $\lambda(1), \ldots, \lambda(2k)$, $k \geq 2$, there exists a linear program $A\mathbf{x} \leq \mathbf{b}$, where the matrix $A$ has entries $0,1,-1,$ $\mathbf{b}$ is a vector of homogeneous, integral, linear forms in terms of the parts of $\lambda(1), \ldots, \lambda(2k)$, and the multiplicity 
\[
\sum c^{\lambda(1)}_{\alpha(1),\alpha(2)} c^{\lambda(2)}_{\alpha(2),\alpha(3)} \cdots c^{\lambda(2k-1)}_{\alpha(2k-1),\alpha(2k)} c^{\lambda(2k)}_{\alpha(2k),\alpha(1)}
\]
is equal to the number of solution vectors $\mathbf{x}$ with integer entries. 
\end{lemma}

With this, we can prove the complexity of the positivity of this multiplicity.

\begin{proof}[Proof of Theorem \ref{sun-complexity}]
First, we claim that the $m$-sun hive polytope, whose number of lattice points equals 
\[
f(\lambda(1),\ldots, \lambda(2k)) = \sum c^{\lambda(1)}_{\alpha(1),\alpha(2)} c^{\lambda(2)}_{\alpha(2),\alpha(3)} \cdots c^{\lambda(2k-1)}_{\alpha(2k-1),\alpha(2k)} c^{\lambda(2k)}_{\alpha(2k),\alpha(1)} 
\]
 contains an (integer) $(2k,n)-$LR sun hive if and only if it is nonempty, which is equivalent to the multiplicity being nonzero. Note that because the polytope is defined by a linear system $A \mathbf{x} \leq \mathbf{b}$ where $\mathbf{b}$ is homogeneous (Lemma \ref{linear-program}), for any integer $N$,  $f(N\lambda(1), \ldots, N \lambda(2k))$ is the number of integer  $(2k,n)$-LR sun hives in the polytope with scaled external boundaries.
 
One direction  of the claim is trivial, so suppose the polytope is nonempty. In particular, the polytope has a vertex. One characterization of a vertex of a polytope (see, for instance, \cite{Sch03}) 
 defined by such a system of inequalities $A \mathbf{x} \leq \mathbf{b}$  is a point $\mathbf{v}$ of the polytope such that $A\mathbf{v} = \mathbf{b}$.  Because $A$ is of full rank (because the defined polytope is nonempty) and the entries of $A$ and $\mathbf{b}$ are all integers, Cramer's rule implies that all the vertices of the polytope have rational coefficients. There is then an integer $N$ for which the scaled polytope contains a $(2k,n)$-LR sun hive. The saturation theorem \ref{saturation-sun}  ensures that $f(\lambda(1), \ldots, \lambda(m))$ is positive, so the original polytope contains a $(2k,n)$-LR sun hive. 

Determining whether the polytope is nonempty or not can be determined in polynomial time using linear programming, such as the ellipsoid or interior point algorithm. Furthermore, because the linear program $A \mathbf{x} \leq \mathbf{b}$ is combinatorial, positivity can be determined in strongly polynomial time by using the algorithm in \cite{Tar86}.
\end{proof}

\section{Other generalized Littlewood-Richardson coefficients }
\label{section-others}
\hypertarget{section-others}{}
 
 The same techniques used in this paper can be used to prove similar results for two other generalized Littlewood-Richardson coefficients, and as such we state the results for these multiplicities without proof. The details for the calculations and proofs may be found in the author's  PhD thesis. 
 
 \subsection{Context and motivation}
 The two other multiplicities are 
 \begin{equation}\label{two}
f_1(\lambda(1),\ldots, \lambda(m)) :=
\sum c^{\alpha(1)}_{\lambda(1),\lambda(2)}   c^{\lambda(3)}_{\alpha(1),\alpha(2)}\cdots c^{\lambda(m-2)}_{\alpha(m-4),\alpha(m-3)}  c^{\alpha(m-3)}_{\lambda(m-1), \lambda(m)}
\end{equation}
for $m\geq 4$, and 
\begin{equation}\label{three}
f_2(\lambda(1),\ldots, \lambda(m)):=
\sum c^{\lambda(2)}_{\lambda(1), \alpha(1)}   c^{\lambda(3)}_{\alpha(1), \alpha(2)} \cdots c^{\lambda(m-2)}_{\alpha(m-4), \alpha(m-3)}   c^{\lambda(m-1)}_{\alpha(m-3), \lambda(m)}, 
\end{equation}
for m $\geq 3$,
 where the case $m=3$ is the Littlewood-Richardson coefficient $c^{\lambda(2)}_{\lambda(1),\lambda(3)}$, and  each summation ranges over all partitions $\alpha(i)$. The first multiplicity describes the branching rule for the direct sum embedding $\gl(n) \times \gl(n') \subseteq \gl(n+n')$ when $m=6$. This was first proven in \cite{Kin70}, and is also proven in \cite{HTW05} (see also \cite{HJ09} and \cite{Koi89}). The second multiplicity  describes the tensor product multiplicities for extremal weight crystals of type $A_{+\infty}$, using a combinatorial rule found by Kashiwara \cite{Kas90} similar to the Littlewood-Richardson rule that described the irreducible components of the tensor product of irreducible representations of the quantized universal enveloping algebra of a symmetrizable Kac-Moody algebra $\mathfrak{g}$ as described in \cite{Kwo09} (see also \cite{Kwo10}) again when $m=6$. This generalized multiplicity is described in \cite{Chi08} and \cite{Chi09}, and is  found to have connections with long exact sequences of finite, abelian $p$-groups, parabolic affine Kazhdan-Lusztig polynomials, and decomposition numbers for $q$-Schur algebras.
 
 \subsection{Statement of results}
 
 The quiver representing multiplicity (\ref{two}) is the generalized star quiver used to describe multiplicity (\ref{three}) in \cite{Chi08}, except that the first two flags are oriented in the same direction and likewise the last two are going the same direction. The dimension vector $\beta$ is defined like before as $\beta(j,i) = j$ and the weight is defined similar to $\sigma_1$ in equation (\ref{weight-sigma_1}). Similar calculations to the ones in Section \ref{section-saturation} prove the following.
 
 \begin{theorem}[Saturation property]
 Let $\lambda(1),\ldots, \lambda(m)$, $m \geq 4$, be weakly decreasing sequences of $n$ integers. For every integer $r \geq 1$,
 \[
 f_1(r\lambda(1),\ldots, r\lambda(m)) \neq 0 \Longleftrightarrow f_1(\lambda(1),\ldots, \lambda(m)) \neq 0.
 \]
 \end{theorem}
 
 The saturation for multiplicity (\ref{three}) is Theorem 1.4 of \cite{Chi08}. In the same paper, Chindris provides the Horn-type inequalities and a generalization of Horn's conjecture for the second multiplicity in Theorem 1.6.
 
 The corresponding inequalities and a generalization of Horn's conjecture for the first multiplicity above are slightly complicated by the fact that they depend on the parity of $m$. We state below only the results for $m$ even and remark that similar results hold true for $m$ odd. 
 
 \textbf{Generalized eigenvalue problem for $f_1$.}
For which weakly decreasing sequences $\lambda(1),\ldots, \lambda(2k),\, k\geq 2$, of $n$ real numbers do there exist $n \times n$ complex Hermitian matrices $H(1),\ldots, H(2k)$ with eigenvalues $\lambda(1),\ldots, \lambda(2k)$ such that 
\[
\dsp H(1) + \sum_{i=1}^{k-1} H(2i) = H(2k) + \sum_{i=1}^{k-1} H(2i+1)  
\]
 and such that 
\[
\begin{array}{rcl}
\dsp H(1) + H(2), & \dsp (-1)^j(H(1)+H(2)) + \sum_{i=3}^j (-1)^{j+i} H(i) & 3 \leq j \leq 2k-2
\end{array}
\]
have non-negative eigenvalues?

We define the generalized Klyachko's cone for this multiplicity as the rational convex polyhedral cone of $m$-tuples $(\lambda(1),\ldots, \lambda(m))$ of solutions to this generalized eigenvalue problem. We denote this cone as $K_1(n,m) \subseteq \RR^{nm}$.

For subsets $I_i \subseteq \{1,\ldots, n\}, \, 1 \leq i \leq 2k$, define the following tuple of weakly decreasing sequences of integers:
\[
\unlam_1(I_i) = \begin{cases}
\lambda'(I_i) & 1 \leq i \leq 2k-3, \; i \text{ odd} \\
\lambda'(I_2) - ((|I_2| - |I_3|)^{n-|I_2|}) & i=2\\
\lambda'(I_i) - ((|I_i| - |I_{i-1}| - |I_{i+1}|)^{n-|I_i|}) & 4 \leq i \leq 2k-2, \; i \text{ even} \\
\lambda'(I_{2k-1}) - ((|I_{2k-1}| -  |I_{2k-1} \backslash \{n\} - |I_{2k} \backslash \{n\}|)^{n-|I_{2k-1}|}) & i=2k-1 \\
\lambda'(I_{2k} \backslash \{n\}) &  i=2k.
\end{cases}
\]
 
 The generalization of Horn's conjecture for multiplicity (\ref{two}) may then be stated as follows (for $m$ even).
 
 \begin{theorem}
Let $\lambda(1),\ldots, \lambda(m),\, m=2k \geq 4$, be weakly decreasing sequences of $n$ real numbers. Then the following conditions are equivalent:
		\begin{enumerate}
		\item $(\lambda(1),\ldots, \lambda(m)) \in K_1(n,m)$;
		\item the numbers $\lambda(i)_j$ satisfy
		\[
		\sum_{i=1}^{k-1} |\lambda(2i+1)| + |\lambda(2k)|  =
		|\lambda(1)| + \sum_{i=1}^{k-1} |\lambda(2i)| 
		\]
		and
		\[
		\sum_{j \in I_i} \sum_{i=1}^{k-1} \lambda(2i+1)_j+ \sum_{j \in I_{2k}} \lambda(2k)_j \leq
		\sum_{j \in I_1} \lambda(1)_j + \sum_{j \in I_i} \sum_{i=1}^{k-1} \lambda(2i)_j
		\]
		for every tuple $(I_1, \ldots, I_{m})$ for which $|I_i|<n$ for some $i$, $|I_1| = |I_2|$ and $|I_{m-1}| = |I_m|$, $\unlam_1(I_i)$ are  partitions, $1 \leq i \leq m$,  and 
		\[
		f_1(\unlam_1(I_1),\ldots, \unlam_1(I_m)) \neq 0.
		\]
		If the $\lambda(i)$ are sequences of integers, conditions $(1)-(2)$ are equivalent to 
		\item $f_1(\lambda(1),\ldots, \lambda(m)) \neq 0$.
		\end{enumerate}
\end{theorem}

 Note that in this case we get an equivalence of the tuples that satisfy the generalized eigenvalue problem and those that satisfy the Horn-type inequalities, as opposed to the case for the sun quiver. This again provides a recursive method for finding all nonzero generalized Littlewood-Richardson coefficients of this type. Like before, the list may be shortened.
 
 \begin{proposition}
 The following statements are true for $m=2k \geq 4$:
 \begin{enumerate}
 \item We have $\dm K_1(n,m) = mn-1.$
 \item The cone $K_1(n,m)$ consists of all tuples $(\lambda(1),\ldots, \lambda(m))$ of weakly decreasing sequences of $n$ reals such that 
 \[
		\sum_{i=1}^{k-1} |\lambda(2i+1)| + |\lambda(2k)|  =
		|\lambda(1)| + \sum_{i=1}^{k-1} |\lambda(2i)| 
		\]
		and
		\[
		\sum_{j \in I_i} \sum_{i=1}^{k-1} \lambda(2i+1)_j+ \sum_{j \in I_{2k}} \lambda(2k)_j \leq
		\sum_{j \in I_1} \lambda(1)_j + \sum_{j \in I_i} \sum_{i=1}^{k-1} \lambda(2i)_j
		\]
		for every tuple $(I_1, \ldots, I_{m})$ for which $|I_i|<n$ for some $i$, $|I_1| = |I_2|$ and $|I_{m-1}| = |I_m|$, $\unlam_1(I_i)$ are  partitions, $1 \leq i \leq m$,  and 
		\[
		f_1(\unlam_1(I_1),\ldots, \unlam_1(I_m)) = 1.
		\]
\end{enumerate}
\end{proposition}
 
 Factorization formulas can be found for these multiplicities by using these Horn-type inequalities. Likewise, we can define level-1 weights for the corresponding quivers for (\ref{two}) and (\ref{three}), yet these prove to be less interesting than for (\ref{one}): if a level-1 weight is effective in either case, then the Littlewood-Richardson polynomial has value one and stretching it by a factor of $N$ does not change this. This at least shows that Fulton's conjecture remains true for these multiplicities.
 
 A similar method of gluing LR hives together produces a polytopal description of these multiplicities. In fact, the two have the same generalized LR hives (differing in shape depending on the parity of $m$), except the boundary labels for (\ref{two}) use the related partitions $\widetilde{\lambda(i)}$ for $i=1,2,m-1,m$. Because of this polytopal description and the saturation properties, we again can determine the complexity of their positivity.
 
 \begin{theorem}
 Determining whether multiplicities \emph{(\ref{two})} and \emph{(\ref{three})} are positive or not can be decided in polynomial time. Even more, each can be decided in strongly polynomial time in the sense of \cite{Tar86}.
 \end{theorem}

 \section{Appendix}
\label{appendix}
The following are the dimension vectors $\beta_1$  that correspond to a complete and minimal list of Schur roots $\beta_1$ and $\beta_2$ such that $\beta = \beta_1 + \beta_2$ and $\beta_1 \circ \beta_2 = 1$ in the case $(n,m) =(2,6)$, up to permutations of the flags respecting the symmetries of the sun quiver.

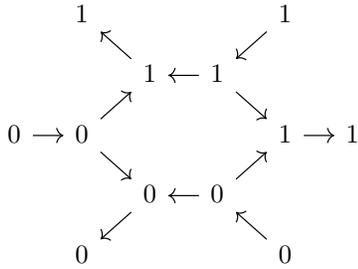
\[
\begin{tikzcd}[column sep=small, row sep=small]
  &1&&& 1 \arrow[dl] \\
&& 1\arrow[ul] & 1 \arrow[l] \arrow[dr] \\
 0\arrow[r] & 0 \arrow[ur] \arrow[dr]&&& 1 \arrow[r]& 1 \\
&& 0 \arrow[dl] & 0 \arrow[l] \arrow[ur] \\
&0&&& \arrow[ul] 0
\end{tikzcd} 
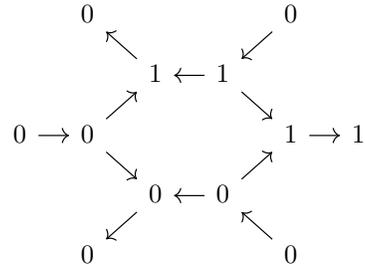
\hspace{2in}
\begin{tikzcd}[column sep=small, row sep=small]
 &0&&& 0 \arrow[dl] \\
&& 1 \arrow[ul] & 1\arrow[l] \arrow[dr] \\
 0\arrow[r] & 0 \arrow[ur] \arrow[dr]&&& 1  \arrow[r]& 1 \\
&& 0 \arrow[dl] & 0 \arrow[l] \arrow[ur] \\
&0&&& \arrow[ul] 0
\end{tikzcd} 
\]

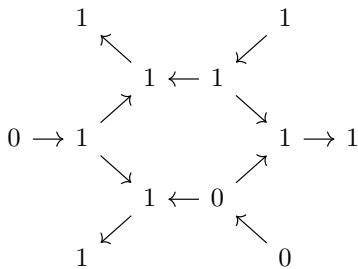
\[
\begin{tikzcd}[column sep=small, row sep=small]
  &1&&& 1 \arrow[dl] \\
&& 1\arrow[ul] & 1 \arrow[l] \arrow[dr] \\
 0\arrow[r] & 1 \arrow[ur] \arrow[dr]&&& 1 \arrow[r]& 1 \\
&& 1 \arrow[dl] & 0 \arrow[l] \arrow[ur] \\
&1&&& \arrow[ul] 0
\end{tikzcd} 
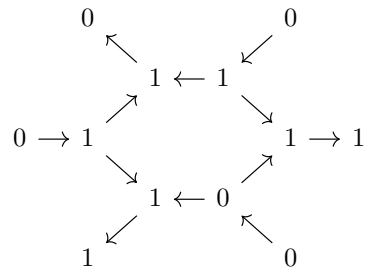
\hspace{2in}
\begin{tikzcd}[column sep=small, row sep=small]
 &0&&& 0 \arrow[dl] \\
&& 1 \arrow[ul] & 1\arrow[l] \arrow[dr] \\
 0\arrow[r] & 1 \arrow[ur] \arrow[dr]&&& 1  \arrow[r]& 1 \\
&& 1 \arrow[dl] & 0 \arrow[l] \arrow[ur] \\
&1&&& \arrow[ul] 0
\end{tikzcd} 
\]

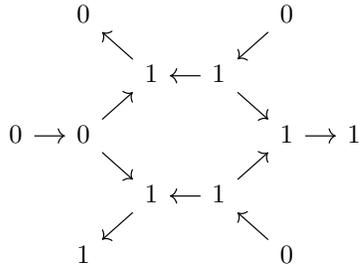
\[
\begin{tikzcd}[column sep=small, row sep=small]
  &0&&& 0 \arrow[dl] \\
&& 1\arrow[ul] & 1 \arrow[l] \arrow[dr] \\
 0\arrow[r] & 0 \arrow[ur] \arrow[dr]&&& 1 \arrow[r]& 1 \\
&& 1 \arrow[dl] & 1 \arrow[l] \arrow[ur] \\
&1&&& \arrow[ul] 0
\end{tikzcd} 
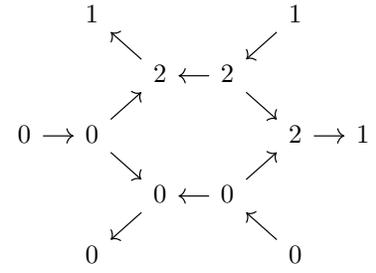
\hspace{2in}
\begin{tikzcd}[column sep=small, row sep=small]
 &1&&& 1 \arrow[dl] \\
&& 2 \arrow[ul] & 2\arrow[l] \arrow[dr] \\
 0\arrow[r] & 0 \arrow[ur] \arrow[dr]&&& 2  \arrow[r]& 1 \\
&& 0 \arrow[dl] & 0 \arrow[l] \arrow[ur] \\
&0&&& \arrow[ul] 0
\end{tikzcd} 
\]

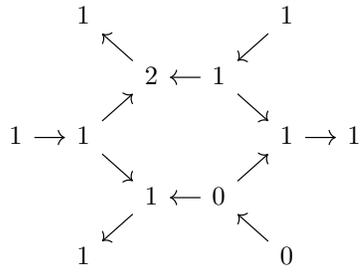
\[
\begin{tikzcd}[column sep=small, row sep=small]
  &1&&& 1 \arrow[dl] \\
&& 2\arrow[ul] & 1 \arrow[l] \arrow[dr] \\
 1\arrow[r] & 1 \arrow[ur] \arrow[dr]&&& 1 \arrow[r]& 1 \\
&& 1 \arrow[dl] & 0 \arrow[l] \arrow[ur] \\
&1&&& \arrow[ul] 0
\end{tikzcd} 
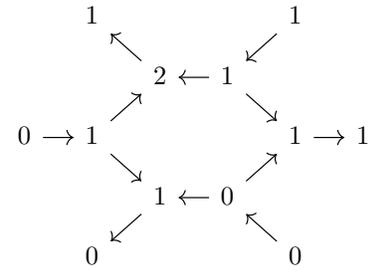
\hspace{2in}
\begin{tikzcd}[column sep=small, row sep=small]
 &1&&& 1 \arrow[dl] \\
&& 2 \arrow[ul] & 1\arrow[l] \arrow[dr] \\
 0\arrow[r] & 1 \arrow[ur] \arrow[dr]&&& 1  \arrow[r]& 1 \\
&& 1 \arrow[dl] & 0 \arrow[l] \arrow[ur] \\
&0&&& \arrow[ul] 0
\end{tikzcd} 
\]

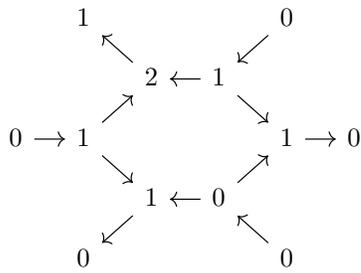
\[
\begin{tikzcd}[column sep=small, row sep=small]
  &1&&& 0 \arrow[dl] \\
&& 2 \arrow[ul] & 1 \arrow[l] \arrow[dr] \\
 0\arrow[r] & 1 \arrow[ur] \arrow[dr]&&& 1 \arrow[r]& 0 \\
&& 1 \arrow[dl] & 0 \arrow[l] \arrow[ur] \\
&0&&& \arrow[ul] 0
\end{tikzcd} 
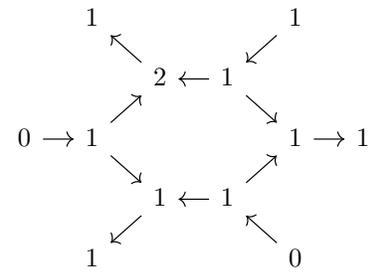
\hspace{2in}
\begin{tikzcd}[column sep=small, row sep=small]
 &1&&& 1 \arrow[dl] \\
&& 2 \arrow[ul] & 1\arrow[l] \arrow[dr] \\
 0\arrow[r] & 1 \arrow[ur] \arrow[dr]&&& 1  \arrow[r]& 1 \\
&& 1 \arrow[dl] & 1 \arrow[l] \arrow[ur] \\
&1&&& \arrow[ul] 0
\end{tikzcd} 
\]

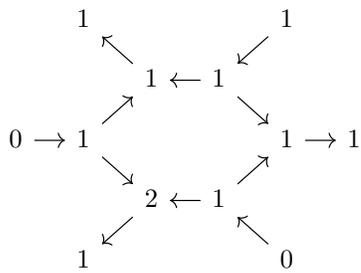
\[
\begin{tikzcd}[column sep=small, row sep=small]
  &1&&& 1 \arrow[dl] \\
&& 1\arrow[ul] & 1 \arrow[l] \arrow[dr] \\
 0\arrow[r] & 1 \arrow[ur] \arrow[dr]&&& 1 \arrow[r]& 1 \\
&& 2  \arrow[dl] & 1 \arrow[l] \arrow[ur] \\
&1&&& \arrow[ul] 0
\end{tikzcd} 
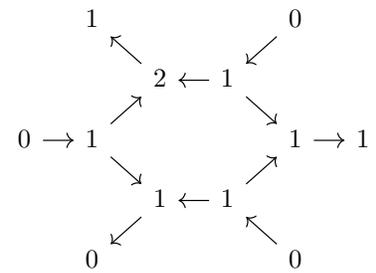
\hspace{2in}
\begin{tikzcd}[column sep=small, row sep=small]
 &1&&& 0 \arrow[dl] \\
&& 2 \arrow[ul] & 1\arrow[l] \arrow[dr] \\
 0\arrow[r] & 1 \arrow[ur] \arrow[dr]&&& 1  \arrow[r]& 1 \\
&& 1 \arrow[dl] & 1 \arrow[l] \arrow[ur] \\
&0&&& \arrow[ul] 0
\end{tikzcd} 
\]

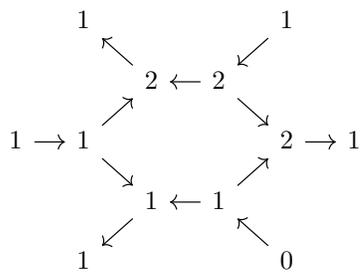
\[
\begin{tikzcd}[column sep=small, row sep=small]
  &1&&& 1 \arrow[dl] \\
&& 2\arrow[ul] & 2 \arrow[l] \arrow[dr] \\
 1\arrow[r] & 1 \arrow[ur] \arrow[dr]&&& 2 \arrow[r]& 1 \\
&& 1 \arrow[dl] & 1 \arrow[l] \arrow[ur] \\
&1&&& \arrow[ul] 0
\end{tikzcd} 
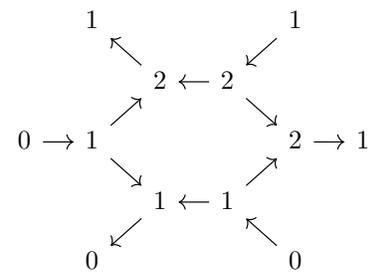
\hspace{2in}
\begin{tikzcd}[column sep=small, row sep=small]
 &1&&& 1 \arrow[dl] \\
&& 2 \arrow[ul] & 2\arrow[l] \arrow[dr] \\
 0\arrow[r] & 1 \arrow[ur] \arrow[dr]&&& 2  \arrow[r]& 1 \\
&& 1 \arrow[dl] & 1 \arrow[l] \arrow[ur] \\
&0&&& \arrow[ul] 0
\end{tikzcd} 
\]
\section*{Acknowledgement}

The author would like to thank Calin Chindris for his support during the duration of this project, especially suggesting the problem and the many helpful discussions.

The author is also extremely grateful to Velleda Baldoni, Mich\`{e}le Vergne, and Michael Walter for their interest in this paper and for pointing out mistakes in the calculations in Example \ref{n=2} in an earlier draft. They corroborated the inequalities as they are now stated using the techniques developed in \cite{BVW18}.

\medskip
\smallskip
\address{\textsc{Mathematics Department, Fitchburg State University, 160 Pearl St, MA, 01420}}
\par \nopagebreak
\noindent\email{\textit{Email address}: \texttt{bcolli15@fitchburgstate.edu}}

\end{document}